\documentclass[11pt,reqno]{article}

\usepackage[utf8]{inputenc}
\usepackage[english]{babel}

\usepackage{csquotes}

\usepackage{geometry}  
\usepackage{graphicx}
\usepackage{amssymb}
\usepackage{epstopdf}
\usepackage{amsmath}
\usepackage{amssymb}
\usepackage{epsfig}
\usepackage{graphicx}
\usepackage{changebar}
\usepackage{centernot}
\usepackage{rotating}
\usepackage{amsthm}
\usepackage{caption}
\usepackage{rotating}
\usepackage{tikz}
\usepackage{setspace}
\usepackage{subfigure}
\usepackage{changepage}
\usepackage{mathtools}
\usepackage{mathrsfs}
\usepackage{youngtab}
\usepackage{color}
\usepackage{mdframed}
\usepackage{bm}
\usepackage{fancyhdr}
\usepackage{lastpage}
\usepackage{txfonts}
\usepackage{slashed}
\usepackage{comment}
\usepackage{longtable}
\usepackage[hang,flushmargin]{footmisc}

\usepackage{chngcntr}
\counterwithin{figure}{section}
\counterwithin{table}{section}
\counterwithin{equation}{section}

\usepackage[utf8]{inputenc}
\usepackage{lipsum}
\usepackage[inline, shortlabels]{enumitem}
\setlist{itemjoin ={,\enspace},itemjoin* = {\enspace}}

\usepackage{hyperref}
\hypersetup{
    linkcolor=blue,
    citecolor=blue,
}

\allowdisplaybreaks

\usetikzlibrary{calc}

\makeatletter
\newcommand*\bigcdot{\mathpalette\bigcdot@{.5}}
\newcommand*\bigcdot@[2]{\mathbin{\vcenter{\hbox{\scalebox{#2}{$\m@th#1\bullet$}}}}}
\makeatother

\newenvironment{myfont}{\fontfamily{phv}\selectfont}{\par}
\def\beg{\begin{myfont}}
\def\en{\end{myfont}}

\usetikzlibrary{positioning}

\def\mc{\mathcal}
\def\mb{\mathbb}
\def\ni{\noindent}

\def\tb{\textbf}

\def\t{\text}

\def\la{\langle}
\def\ra{\rangle}
\def\mf{\mathfrak}
\def\vs{\vspace{.1in}}

\def\la{\langle}
\def\ra{\rangle}
\def\bp{\begin{proof}}
\def\ep{\end{proof}}

\def\be{\begin{enumerate}}
\def\ee{\end{enumerate}}

\def\adj{\sim}

\def\dist{\nsimeq}
\def\bi{\begin{itemize}}
\def\ei{\end{itemize}}
\def\i{\item[]}

\def\multiset#1#2{\ensuremath{\left(\kern-.3em\left(\genfrac{}{}{0pt}{}{#1}{#2}\right)\kern-.3em\right)}}

\newtheorem{Lem}{Lemma}[section]
\newtheorem{Thm}[Lem]{Theorem}
\newtheorem{Cor}[Lem]{Corollary}

\newtheorem{Prop}[Lem]{Proposition}

\theoremstyle{definition}
\newtheorem{Def}[Lem]{Definition}
\newtheorem{Rmk}[Lem]{Remark}

\title{Unified characterizations of minuscule Kac--Moody representations built from colored posets\thanks{This paper is part of a Ph.D. thesis \cite{Str} written under the supervision of R.A. Proctor at the University of North Carolina at Chapel Hill.}}

\author{Michael C. Strayer \\ University of North Carolina at Chapel Hill \\ Chapel Hill, NC 27599 U.S.A. \\ mcs80@live.unc.edu}
\date{June 17, 2019}                                           

\begin{document}

\maketitle 

\begin{spacing}{1.1}

    
    
    
    
    
    
    
    
    
    



\begin{abstract}
R.M. Green described structural properties that ``doubly infinite'' colored posets should possess so that they can be used to construct representations of most affine Kac--Moody algebras.  These representations are analogs of the minuscule representations of the semisimple Lie algebras, and his posets (``full heaps'') are analogs of the finite minuscule posets.  Here only simply laced Kac--Moody algebras are considered.  Working with their derived subalgebras, we provide a converse to Green's theorem.  Smaller collections of colored structural properties are also shown to be necessary and sufficient for such
poset-built 
representations 
to be produced for smaller subalgebras, especially the ``Borel derived'' subalgebra.  These developments lead to the formulation of unified definitions of finite and infinite colored minuscule and $d$-complete posets.  This paper launches a program that seeks to extend the notion of ``minuscule representation'' to Kac--Moody algebras, and to classify such representations.  
\end{abstract}



















\end{spacing}


\begin{spacing}{1.31}

\section{Introduction}\label{119}


For most affine Kac--Moody algebras, R.M. Green constructed \cite{Gre1} a small number of beautiful representations whose weight diagrams were unbounded above and below; we refer to such structures as being ``doubly infinite.''  In contrast, the familiar Category $\mc{O}$ representations have weight diagrams that are bounded above.  Green's representations and the doubly infinite posets from which they were built formed a central topic in his 2013 Cambridge tract \cite{Gre}.  He noted that these representations (with no highest weights) are analogous in many ways to the minuscule representations of semisimple Lie algebras (which are finite dimensional with highest weights).  For infinite dimensional Kac--Moody algebras, there are no highest weight representations that are analogous to the minuscule representations of semisimple Lie algebras.  This paper is the first in a series of papers wherein we hope to define an abstract notion of ``minuscule'' representation for arbitrary Kac--Moody algebras and then to classify them; see Section 9.3 of \cite{Str}.

The principal antecedents to this paper are \cite{Gre1} and \cite{Gre}, and important antecedents to those works were \cite{Ste} and \cite{Hag2}.  Those references were concerned with minuscule posets, $d$-complete posets, and full heaps.  In the next three paragraphs we indicate where these kinds of posets first appeared.  From our perspective these posets are used to build three out of the four kinds of representations appearing below in Table \ref{RepChar}.  The unified definitions that we develop for finite and infinite ``$\Gamma$-colored minuscule'' and ``$\Gamma$-colored $d$-complete'' posets may be of more interest than any one of our stated results by itself.  These new overarching definitions do not depend on knowing the cardinality of the poset \emph{a priori}; in addition they fill in the lower right slot of Table \ref{RepChar}. These definitions are presented in Section \ref{668}, wherein our main results are also summarized. That section has been written for immediate accessibility.

It has been known that the minuscule representations of the semisimple Lie algebras can be constructed combinatorially.  Consider one of the irreducible minuscule posets $P$ introduced by Proctor that were colored by him in Theorem 11 of \cite{Pro4} with the nodes of an associated Dynkin diagram $\Gamma$.  Form the set $\mc{FI}(P)$ of all of the ``splits'' $F/I$ of $P$, where $F$ is an upwardly closed subset (filter) of $P$ and $I$ is its complementary downwardly closed subset (ideal) of $P$.  These splits are the elements of a distributive lattice, called an irreducible minuscule lattice in \cite{Pro4}, whose covering edges are colored by the Dynkin nodes.  These edges can be used to define colored raising and lowering actions of the Chevalley generators of the Lie algebra $\mf{g}$ associated to $\Gamma$.  It can be seen that these actions specify a representation of $\mf{g}$ ``carried by'' $\mc{FI}(P)$ with highest weight corresponding to the split $\emptyset / P$.  Wildberger used \cite{Wil} this picture to specify the actions of a Chevalley basis for the Lie algebra in a minuscule representation; see Section 7.2 of \cite{Gre}.  R.G. Donnelly constructed many representations of semisimple Lie algebras using lattices of splits; see \cite{Don} and subsequent papers.
Proctor showed \cite{Pro4} that minuscule posets have the combinatorial Sperner property and (with R. Stanley's help) the combinatorial Gaussian property; see Section 11.3 of \cite{Gre}.  Minuscule posets are the structures on which the Littlewood--Richardson and cohomology calculations for minuscule varieties performed in \cite{BuSa} and its references are based.  Donnelly has recently given \cite{Don2} a new combinatorial characterization of the minuscule lattices of \cite{Pro4} while developing a new version of crystal graphs.




A generalization of minuscule posets appeared after D. Peterson introduced \cite{Car} a special kind of element in the Kac--Moody Weyl group specified by a Dynkin diagram $\Gamma$.  For an integral weight $\lambda$, he defined ``$\lambda$-minuscule'' elements $w$.  When $\Gamma$ is simply laced, Proctor showed \cite{Pro1} that the Bruhat intervals $[e,w]$ are distributive lattices.  When $\lambda$ is dominant, he then characterized the finite $\Gamma$-colored poset $P$ of join irreducibles of the lattice with some structural ``$d$-complete'' conditions.  The reduced decompositions of $w$ corresponded to the linear extensions of $P$.  Working in the context of Viennot's heap for $w$, Stembridge extended \cite{Ste} Proctor's work to include non-simply laced $\Gamma$.  He reformulated Proctor's notion of colored $d$-complete with some elegant coloring axioms, and referred to these posets as ``dominant minuscule heaps.''  More generally, Stembridge characterized the heaps for all $\lambda$-minuscule elements.  M. Hagiwara described \cite{Hag1,Hag2} the minuscule heaps for elements of the Kac--Moody Weyl groups specified by star shaped Dynkin diagrams and for the affine Weyl group of type $\tilde{A}_n$.  Stembridge's coloring axioms for the $d$-complete posets are not all self-dual.  Proctor showed that $d$-complete posets have unique jeu de taquin rectifications \cite{Pro5} and (with D. Peterson's help) the hook length property \cite{Pro3}.  These posets have been receiving increasing attention, as in \cite{KleRam}, and there is a bibliography for them in \cite{ProScop}.  When that study of the axioms for finite uncolored $d$-complete posets was written, it became apparent that the definition of ``$d$-complete'' could likely be extended to infinite locally finite posets.  However, it was unclear precisely what the most appropriate definition should be for such posets.  For further historical details, see Section 13 of \cite{Pro3}.

Adopting some of Stembridge's axioms, Green axiomatically defined \cite{Gre1} ``full heaps'' colored by Dynkin diagrams $\Gamma$.  These are doubly infinite locally finite colored posets $P$ in which the appearances of each color from $\Gamma$ are unbounded above and below.  He regarded these posets as being close companions to the finite minuscule posets.  All of his coloring axioms for full heaps were self-dual.  The ``extended slant lattices'' used by Hagiwara to describe the minuscule heaps for type $\tilde{A}_n$ were early appearances of full heaps.

Let $\mf{g}'$ be the derived Kac--Moody algebra for a Dynkin diagram $\Gamma$.  In our language Green essentially showed that if $P$ is a full heap colored by $\Gamma$, then the lattice of splits $\mathcal{FI}(P)$ carries a representation of $\mathfrak{g}'$.
\ni This result first appeared in Theorem 3.1 of \cite{Gre1} and was restated as Theorem 4.1.6(i) of \cite{Gre}.  
We compare Green's result to the ``sufficient'' direction of our main result Theorem \ref{672}(b) in Remark \ref{GreenComparison}.  For reference, there we state Green's result as Theorem \ref{4.1.6(i)}.
Throughout this paper we restrict our attention to simply laced $\Gamma$.  For such diagrams $\Gamma$, our foremost new main result (the ``necessary'' direction of Theorem \ref{672}(b)) provides a converse to Theorem 
\ref{4.1.6(i)}
that includes finite dimensional representations as well as infinite dimensional representations.  To state this converse, we formulate a notion of ``$P$-minuscule'' representation 
(Definition \ref{MinReps}):
This is a representation of $\mf{g}'$ carried by $\mc{FI}(P)$ that ``looks like'' a minuscule representation of a semisimple Lie algebra.  At the same time, for simply laced $\Gamma$, the ``sufficient'' direction of Theorem \ref{672}(b) is a version of Theorem 
\ref{4.1.6(i)}
that now includes posets of unknown (finite or ``mixed'') cardinality.  
Theorem \ref{672}(b) summarizes Theorem \ref{653}.  

Leading up to Theorem \ref{653} are several intermediate results; as more and more coloring properties are assumed for the poset $P$, the representations constructed have stronger and stronger algebraic properties.  Most often these algebraic properties are the satisfaction of some of the defining relations for $\mf{g}'$.  Each of these collections of coloring properties is necessary as well as sufficient for the collection of algebraic properties at hand.  This development clarifies which collections of the coloring properties assumed for full heaps in Theorem 
\ref{4.1.6(i)}
correspond to which algebraic aspects of the representations.  It also facilitates comparison with the collections of coloring properties considered by Stembridge in his parallel study of the reduced decompositions of a $\lambda$-minuscule Kac--Moody Weyl group element $w$.

By omitting a ``down-only'' coloring property required in Theorem \ref{653}, earlier in Theorem \ref{614} we obtain a similar characterization of colored poset constructions of representations of just the ``Borel derived'' (Section \ref{251}) subalgebra $\mf{b}'_+$.  When $P$ is finite, these representations arise as the restrictions to $\mf{b}'_+$ of the Demazure $\mf{b}_+$-modules for the dominant $\lambda$-minuscule $w$.  (Even when $P$ is finite, the Kac--Moody algebra $\mf{g}$ at hand can be infinite dimensional.)  For this ``up only'' analog to Theorem \ref{653}, we introduce two new definitions.  We formulate the notion of ``upper $P$-minuscule'' representation of $\mf{b}'_+$ (Definition \ref{MinReps}).  For the lower right corner of Table \ref{RepChar}, we formulate a notion of colored $d$-complete that works for infinite locally finite colored posets.  So this desire to obtain a theorem for $\mf{b}'_+$ analogous to Theorem \ref{653} led to a precise definition for infinite colored $d$-complete posets.  Theorem \ref{672}(a) summarizes Theorem \ref{614}.

After introducing full heaps, Green defined \cite{Gre} the notion of ``principal subheap'' for full heaps colored by affine $\Gamma$.  These are finite colored posets.  He showed that the principal subheaps of such a full heap are isomorphic to each other.  Then he proved that the possible principal subheaps are exactly the pre-existing finite colored minuscule posets.  In \cite[\S 8.4]{Str} we showed that Green's full heaps are exactly our infinite $\Gamma$-colored minuscule posets and that Green's principal subheaps (the pre-existing finite colored minuscule posets) are exactly our finite $\Gamma$-colored minuscule posets.  The relationship between the finite $\Gamma$-colored minuscule posets and the infinite $\Gamma$-colored minuscule posets is entirely different here than in \cite{Gre}.  

\vspace{.1in}

\begin{table}[h]
    \centering
    \begin{tabular}{|l||c|c|}
        \hline 
         \textbf{Representation} & \textbf{Finite dimensional} & \textbf{Infinite dimensional} \\
         \hline 
         \hline
        \textbf{$\boldsymbol{P}$-minuscule} & Finite $\Gamma$-colored minuscule posets & Infinite $\Gamma$-colored minuscule posets \\
        \hline 
        \textbf{Upper $\boldsymbol{P}$-minuscule} & Finite $\Gamma$-colored $d$-complete posets & Infinite $\Gamma$-colored $d$-complete posets \\
        \hline
    \end{tabular}
    \caption{The characterizations of $P$-minuscule and upper $P$-minuscule representations}
    \label{RepChar}
\end{table}

To give an overview of this paper, for each simply laced $\Gamma$ we regard Table \ref{RepChar} as a table of representation characterization problems.  (Extensions of our results that also handle the non-simply laced case are presented in \cite{Str}.)  The rows of the table respectively pose existence problems for ``$P$-minuscule'' representations of $\mf{g}'$ and ``upper $P$-minuscule'' representations of $\mf{b}'_+$.  The posets $P$ that are shown to solve these problems are respectively the $\Gamma$-colored minuscule posets and the $\Gamma$-colored $d$-complete posets.  The two columns of the table pertain to the cardinality of the poset $P$.  Earlier work \cite{Pro1,Ste} has either been restricted to finite posets $P$ or has handled \cite{Gre} finite and infinite posets separately.  Using \emph{a posteriori} knowledge of the cardinality of $P$, the columns of the table indicate finite $P$ or infinite $P$; in this paper we do not assume \emph{a priori} (not even in the proofs) that the poset $P$ at hand is known to be finite or infinite.  Our answers to these four existence problems are summarized in Theorem \ref{672}.  Stembridge's dominant minuscule heaps (i.e. Proctor's colored $d$-complete posets) inhabit the lower left corner of Table \ref{RepChar} and Green's full heaps inhabit the upper right corner.  The infinite $\Gamma$-colored $d$-complete posets that inhabit the lower right corner are new.  The original colored minuscule posets inhabit the upper left corner. 

It is the introduction of ``frontier census'' coloring properties (Section \ref{522}) for the poset $P$ that enables us to provide definitions of ``$\Gamma$-colored minuscule'' and ``$\Gamma$-colored $d$-complete'' that are uniform across Table \ref{RepChar}.  Given an element $y \in P$ that is an extreme element of $P$ of color $b \in \Gamma$, these properties limit the number of elements that lie beyond $y$ in $P$ that have colors that are adjacent to $b$ in $\Gamma$.

As is true for the finite minuscule case, when $P$ is a full heap colored by $\Gamma$ of untwisted affine type the structure $\mc{FI}(P)$ can be viewed as a crystal.  Then $\mc{FI}(P)$ can be used to give a representation of the corresponding quantum affine algebra, as is described in Section 8 of Green's paper \cite{Gre1}.  Green also used $\mc{FI}(P)$ to construct representations of most of the affine Weyl groups in \cite{Gre2}. 

Proctor classified \cite{Pro2,Pro1} the finite colored $d$-complete posets for simply laced $\Gamma$.  Stembridge extended \cite{Ste} this classification to non-simply laced $\Gamma$.  Green \cite{Gre} and McGregor-Dorsey \cite{McD} classified the full heaps.  
In \cite[Ch. 8]{Str} we classified the infinite colored $d$-complete posets
and listed all of the posets that are organized by Table \ref{RepChar} above.  
These results will also appear in \cite{ProStr}.
Having the new ``necessary'' direction of Theorem \ref{672}(b) available will enable us to also classify the $P$-minuscule representations.  This will be done by using that direction to combine two of the main results of \cite{Gre} and \cite{McD}, namely Theorem
\ref{4.1.6(i)}
and their classification of full heaps.  This classification of $P$-minuscule representations will be a step in our minuscule Kac--Moody program: In an anticipated third paper we plan to present a definition of ``abstract minuscule'' representation for a Kac--Moody algebra.  This definition will not refer to a poset that has been supplied \emph{a priori}.  Given such a representation, we believe we will be able to construct a poset $P$ so that the given representation can be viewed as a $P$-minuscule representation.

After definitions are given in Sections \ref{236}-\ref{284}, this paper has three parts.  The first part, Sections \ref{317} and \ref{361}, concerns representations of the Borel derived subalgebra $\mf{b}'_+$ that are carried by the lattice of splits $\mc{FI}(P)$.  Theorem \ref{351} states that the possession of three of our earliest coloring properties by the poset $P$ is equivalent to the existence of a representation of the smaller subalgebra $\mf{n}_+ \subset \mf{b}'_+$ that is carried by $\mc{FI}(P)$.  Section \ref{361} studies the extension of this representation from $\mf{n}_+$ to the Borel derived subalgebra $\mf{b}'_+$ by specifying the actions of the simple coroots (which form a basis of the Cartan derived subalgebra $\mf{h}'$).  There is some freedom available for such an extension; the weight functions we introduce are accounting tools to keep track of the coroot actions.  In the second part, Sections \ref{414}-\ref{522}, we introduce a particular nice weight function.  The prototypical minuscule representations of semisimple Lie algebras have weights along their ``$\mf{sl}_2$ strings'' that are composed of eigenvalues from $\{-1,0,+1\}$ for the simple coroot actions.  Our preferred weight function is defined in Definition \ref{Mu}.  In Proposition \ref{458} we begin to obtain simple coroot actions with $\{-1,0,+1\}$ values when three coloring properties for $P$ are present.  We introduce the frontier census coloring properties in Section \ref{522}.  Sections \ref{609}-\ref{668} form our third part.  After obtaining our main results in Sections \ref{609} and \ref{631}, we summarize them in Section \ref{668} by presenting our new definitions of $\Gamma$-colored $d$-complete and $\Gamma$-colored minuscule posets.  We briefly describe their classifications in Remark \ref{ClassifySummary}.


\section{Combinatorial definitions}\label{236}

\begin{figure}[p!]
\centering 
    \begin{tikzpicture}[scale=1]
        \node (f) at (-4,1){$f$};
        \node (e) at (-4,2){$e$};
        \node (g) at (-4,3){$g$};
        \node (a) at (-5,3){$a$};
        \node (b) at (-6,3){$b$};
        \node (c) at (-3,3){$c$};
        \node (d) at (-2,3){$d$};
        
        \draw (b) -- (a) -- (g) -- (c) -- (d);
        \draw (g) -- (e) -- (f);
    
        \node (y) at (-6,14){$y_d$};
        \node (z) at (-4,14){$z_g$};
        \node (w) at (-5,13){$w_c$};
        \node (x) at (-3,13){$x_e$};
        \node (u) at (-4,12){$u_g$};
        \node (v) at (-2,12){$v_f$};
        \node (s) at (-5,11){$s_a$};
        \node (t) at (-3,11){$t_e$};
        \node (q) at (-6,10){$q_b$};
        \node (r) at (-4,10){$r_g$};
        \node (o) at (-5,9){$o_a$};
        \node (p) at (-3,9){$p_c$};
        \node (m) at (-4,8){$m_g$};
        \node (n) at (-2,8){$n_d$};
        \node (k) at (-5,7){$k_e$};
        \node (l) at (-3,7){$l_c$};
        \node (i) at (-6,6){$i_f$};
        \node (j) at (-4,6){$j_g$};
        
        \draw (q) -- (o) -- (r) -- (p) -- (n) -- (l) -- (m) -- (o);
        \draw (i) -- (k) -- (j) -- (l);
        \draw (k) -- (m) -- (p);
        \draw[dashed] (i) -- (-5,5) -- (j) -- (-3,5);
        \draw (y) -- (w) -- (z) -- (x) -- (v) -- (t) -- (r) -- (s) -- (q);
        \draw (w) -- (u) -- (s);
        \draw (x) -- (u) -- (t);
        \draw[dashed] (y) -- (-5,15) -- (z) -- (-3,15);
        
        \node (N) [draw] at (6,2){$n$};
        \node (J)[draw] at (4,0){$j$};
        \node (I) at (2,0){};
        
        \draw[dashed] (J) -- (3,-1);
        
        \node at (.3,13.35){\footnotesize $d$};
        \node at (.3,14.65){\footnotesize $f$};
        \node at (1.3,15.65){\footnotesize $e$};
        \node at (2.3,16.65){\footnotesize $g$};
        \node at (1.3,12.35){\footnotesize $c$};
        \node at (2.3,11.35){\footnotesize $g$};
        \node at (2.3,10.65){\footnotesize $e$};
        \node at (3.7,15.35){\footnotesize $g$};
        \node at (3.7,14.65){\footnotesize $c$};
        \node at (3.7,13.35){\footnotesize $e$};
        \node at (3.7,12.65){\footnotesize $g$};
        \node at (4.7,11.65){\footnotesize $a$};
        \node at (5.7,10.65){\footnotesize $b$};
        \node at (5.7,9.35){\footnotesize $f$};
        \node at (4.7,8.35){\footnotesize $e$};
        \node at (3.7,7.35){\footnotesize $g$};
        \node at (3.7,6.65){\footnotesize $a$};
        \node at (3.7,5.35){\footnotesize $c$};
        \node at (3.7,4.65){\footnotesize $g$};
        \node at (4.7,3.65){\footnotesize $e$};
        \node at (5.7,2.65){\footnotesize $f$};
        \node at (5.7,1.35){\footnotesize $d$};
        \node at (4.7,0.35){\footnotesize $c$};
        \node at (2.3,1.35){\footnotesize $e$};
        \node at (2.3,2.65){\footnotesize $c$};
        \node at (2.3,3.35){\footnotesize $g$};
        \node at (1.3,4.35){\footnotesize $a$};
        \node at (.3,5.35){\footnotesize $b$};
        \node at (.3,6.65){\footnotesize $d$};
        \node at (1.3,7.65){\footnotesize $c$};
        \node at (2.3,8.65){\footnotesize $g$};
        \node at (2.3,9.35){\footnotesize $a$};
        
        \node (YZ)[draw] at (3,17){$y,z$};
        \node (YX)[draw] at (2,16){$y,x$};
        \node (YV)[draw] at (1,15){$y,v$};
        \node (Y)[draw] at (0,14){$y$};
        \node (W)[draw] at (1,13){$w$};
        \node (U)[draw] at (2,12){$u$};
        \node (ST)[draw] at (3,11){$s,t$};
        \node (S)[draw] at (2,10){$s$};
        \node (Z)[draw] at (4,16){$z$};
        \node (WX)[draw] at (3,15){$w,x$};
        \node (WV)[draw] at (2,14){$w,v$};
        \node (X)[draw] at (4,14){$x$};
        \node (UV)[draw] at (3,13){$u,v$};
        \node (SV)[draw] at (4,12){$s,v$};
        \node (QV)[draw] at (5,11){$q,v$};
        \node (QT)[draw] at (4,10){$q,t$};
        \node (V)[draw] at (6,10){$v$};
        \node (T)[draw] at (5,9){$t$};
        \node (QR)[draw] at (3,9){$q,r$};
        \node (R)[draw] at (4,8){$r$};
        \node (QP)[draw] at (2,8){$q,p$};
        \node (OP)[draw] at (3,7){$o,p$};
        \node (QN)[draw] at (1,7){$q,n$};
        \node (Q)[draw] at (0,6){$q$};
        \node (ON)[draw] at (2,6){$o,n$};
        \node (P)[draw] at (4,6){$p$};
        \node (O)[draw] at (1,5){$o$};
        \node (MN)[draw] at (3,5){$m,n$};
        \node (M)[draw] at (2,4){$m$};
        \node (KN)[draw] at (4,4){$k,n$};
        \node (KL)[draw] at (3,3){$k,l$};
        \node (IN)[draw] at (5,3){$i,n$};
        \node (K)[draw] at (2,2){$k$};
        \node (IL)[draw] at (4,2){$i,l$};
        \node (IJ)[draw] at (3,1){$i,j$};
        \node (L)[draw] at (5,1){$l$};
        
        \draw (S) -- (QR) -- (QT) -- (T) -- (V);
        \draw (T) -- (R) -- (QR) -- (QP) -- (QN) -- (Q) -- (O) -- (ON) -- (OP) -- (R);
        \draw (QP) -- (OP) -- (P) -- (MN) -- (ON) -- (QN);
        \draw (O) -- (M) -- (MN) -- (KN) -- (IN) -- (N) -- (L) -- (IL) -- (KL) -- (M);
        \draw (IJ) -- (J) -- (L);
        \draw[dashed] (2,0) -- (IJ);
        \draw (KN) -- (KL) -- (K) -- (IJ) -- (IL) -- (IN);
        \draw (YZ) -- (YX) -- (YV) -- (Y) -- (W) -- (U) -- (ST) -- (S);
        \draw (YZ) -- (Z) -- (WX) -- (WV) -- (UV) -- (X) -- (WX);
        \draw (YX) -- (WX);
        \draw (YV) -- (WV) -- (W);
        \draw (U) -- (UV) -- (SV) -- (ST) -- (QT) -- (QV) -- (V);
        \draw (SV) -- (QV);
        \draw[dashed] (2,18) -- (YZ) -- (4,18);
        \draw[dashed] (Z) -- (5,17);
    \end{tikzpicture}
    \vs 
    \caption{Clockwise from bottom left: Simple graph $\Gamma$ of affine type $\tilde{E}_6$, full heap poset $P$ colored by $\Gamma$, and edge-colored lattice of filter-ideal splits $\mc{FI}(P)$.  Subscripts in $P$ indicate element colors.  Splits $(F,I)$ are described by the maximal elements of the ideal $I$.  Opposite diamond edges in $\mc{FI}(P)$ have the same color.}
    \label{134}
\end{figure}


Fix a partially ordered set $P$ throughout.  Letters such as $z,y,x,\dots$ are used to denote elements of $P$.  We follow Stanley \cite{Sta} for the following terminology: comparable elements, covering relations and the Hasse diagram, closed and open intervals, connected posets, and direct sums of posets.  We assume each $P$ is \emph{locally finite}; this means that all of its closed intervals are finite.  We write $x \to y$ to indicate that $x$ is covered by $y$.  We say $x$ and $y$ are \emph{neighbors} in $P$ if $x \to y$ or $y \to x$.  A subset $F \subseteq P$ is a \emph{filter} of $P$ if whenever $x \in F$ and $y \ge x$, we also have $y \in F$.  Dually, a subset $I \subseteq P$ is an \emph{ideal} of $P$ if whenever $x \in I$ and $y \le x$, we also have $y \in I$.  For each filter $F$ of $P$ there is a corresponding ideal $I := P - F$.  Let $\mc{FI}(P)$ be the set of all ordered pairs $(F,I)$ such that $F$ is a filter of $P$ and $I$ is its corresponding ideal: These are the \emph{splits} of $P$.  The set $\mc{FI}(P)$ becomes a distributive lattice when it is ordered by inclusion of the ideals within the splits.  Figures \ref{134} and \ref{289} display posets $P$ and their lattices $\mc{FI}(P)$ of splits.  We write $(F+x+\dots,I-x-\dots)$ instead of $(F \cup \{x,\dots\},I-\{x,\dots\})$ and $(F-x-\dots,I+x+\dots)$ instead of $(F-\{x,\dots\},I \cup \{x,\dots\})$.  Each edge in the Hasse diagram of $\mc{FI}(P)$ can be viewed as transferring a minimal element of some split's filter to its ideal, where it becomes a maximal element: In $\mc{FI}(P)$ one has $(F,I) \to (F-x,I+x)$ when $x$ is a minimal element of $F$.  Dually, one has $(F+y, I-y) \to (F,I)$ when $y$ is a maximal element of $I$.

Fix a finite simple graph $\Gamma$ throughout, meaning that no loops or multiple edges are allowed.  We will use the symbol $\Gamma$ to also denote its set of vertices.  Letters such as $a,b,c,\dots$ are used to denote vertices of $\Gamma$, which we call \emph{colors}.  A \emph{$\Gamma$-set} is any set whose elements are indexed by the colors in $\Gamma$.
Let $a,b \in \Gamma$.  If $\{a,b\}$ is an edge of $\Gamma$, we write $a \adj b$ and say $a$ and $b$ are \emph{adjacent}.  If $a \ne b$ and $\{a,b\}$ is not an edge of $\Gamma$, we write $a \dist b$ and say $a$ and $b$ are \emph{distant}.  Let $\delta_{ab}$ be the Kronecker delta.  
Define $\theta_{ab} := 2\delta_{ab} - \sum_{c \adj b} \delta_{ac}$.  We have $\theta_{ab} = 2$ if $a = b$ and $\theta_{ab} = -1$ if $a \adj b$ and $\theta_{ab} = 0$ if $a \dist b$.

We equip $P$ with a surjective \emph{coloring function} $\kappa : P \to \Gamma$, and we say that $P$ is a \emph{$\Gamma$-colored} poset.  See Figures \ref{134} and \ref{289}.   For each $a \in \Gamma$, let $P_a := \kappa^{-1}(a)$ be the subset of all elements in $P$ of color $a$.  The coloring of $P$ induces an edge coloring of the Hasse diagram of $\mc{FI}(P)$: The color of an edge is given by the color of the element transferred along that edge.

Various poset coloring properties will be precisely defined as needed; Table \ref{245} indexes these forthcoming definitions.  The poset displayed in Figure \ref{134} satisfies all of these properties, and the poset displayed in Figure \ref{289} satisfies all of them except Mn1LA.


\begin{table}[h!]
    \centering
    \begin{tabular}{|l|l|l|}
        \hline 
        \tb{Property:} \  & \ \tb{Abbreviated definition:} & \ \tb{Location:} \\
        \hline 
        \hline 
        EC & \ elements with Equal colors are Comparable & \ Proposition \ref{323} \  \\
        \hline 
        ND & \ Neighbors have Different colors & \ Proposition \ref{323} \\
        \hline
        NA & \ Neighbors have Adjacent colors & \ Lemma \ref{341} \\
        \hline 
        I3ND & \ Interval of 3 Neighbors has 3 Different colors & \ Lemma \ref{346} \\
        \hline 
        AC & \ elements with Adjacent colors are Comparable & \ Proposition \ref{423} \\
        \hline 
        I2A & \ consecutive color Intervals contain 2 Adjacent colors \ & \ Proposition \ref{423} \\
        \hline 
        Mx$k$GA & \ color Max has $\le k$ elts Greater than it w/ Adjacent colors & \ Section \ref{522} \\
        \hline 
        Mn$k$LA & \ color Min has $\le k$ elts Less than it w/ Adjacent colors & \ Section \ref{522} \\
        \hline 
    \end{tabular}
    \caption{Coloring property abbreviations and locations of definitions}
    \label{245}
\end{table}

\section{Algebraic definitions}\label{251} 

We regard the graph $\Gamma$ as being a \emph{simply laced Dynkin diagram}.  Once $\Gamma$ has been given a total ordering, the associated generalized Cartan matrix is $[\theta_{ab}]_{a,b \in \Gamma}$.  
The entries $\theta_{ab}$ of this matrix are the integers defined in the previous section; they agree with the entries of the generalized Cartan matrix for $\Gamma$ in \cite{Kac}.  
The \emph{Kac--Moody algebra} $\mf{g}$ with \emph{Cartan subalgebra} $\mf{h}$ and subalgebras $\mf{n}_+$ and $\mf{n}_-$ are defined in \cite{Kac} for a given Dynkin diagram $\Gamma$.  The \emph{positive} and \emph{negative Borel subalgebras} are respectively $\mf{b}_\pm := \mf{h} + \mf{n}_\pm$.  We refer to the subalgebras of the \emph{derived subalgebra} $\mf{g}' := [\mf{g},\mf{g}]$ formed by intersections of $\mf{h}$ and $\mf{b}_\pm$ with $\mf{g}'$ as the \emph{Cartan derived subalgebra} $\mf{h}'$ and the \emph{Borel derived subalgebras} $\mf{b}'_\pm$.  The algebras $\mf{n}_{\pm}$ and the derived algebras are generated \cite[\S 9.11]{Kac} by subsets of the symbols $\{x_a,y_a,h_a\}_{a \in \Gamma}$ subject to some defining relations.  When $\Gamma$ is simply laced, the relations are certain unions (specified below) of the following sets:
\renewcommand{\arraystretch}{0.9}
\begin{longtable}{rrl}
(XX): & (i) & $[x_b,x_a] = 0$ if $a,b \in \Gamma$ with $a \dist b$, \\ 
 & (ii) & $[x_a,[x_a,x_b]] = 0$ for all $a,b \in \Gamma$, \\[.1in]
 (YY): & (i) & $[y_b,y_a] = 0$ if $a,b \in \Gamma$ with $a \dist b$, \\
 & (ii) & $[y_a,[y_a,y_b]] = 0$ for all $a,b \in \Gamma$, \\[.1in]
 (HH): & (i) & $[h_b,h_a] = 0$ for all $a,b \in \Gamma$, \\[.1in]
 (HX): & (i) & $[h_a,x_a] = 2x_a$ for all $a \in \Gamma$, \\
  & (ii) & $[h_b,x_a] = -x_a$ if $a,b \in \Gamma$ with $a \adj b$, \\
  & (iii) & $[h_b,x_a] = 0$ if $a,b \in \Gamma$ with $a \dist b$, \\[.1in]
(HY): & (i) & $[h_a,y_a] = -2y_a$ for all $a \in \Gamma$, \\
 & (ii) & $[h_b,y_a] = y_a$ if $a,b \in \Gamma$ with $a \adj b$, \\
 & (iii) & $[h_b,y_a] = 0$ if $a,b \in \Gamma$ with $a \dist b$, \\[.1in]
(XY): & (i) &  $[x_a,y_a] = h_a$ for all $a \in \Gamma$, \\
 & (ii) & $[x_b,y_a] = 0$ if $a,b \in \Gamma$ with $a \ne b$.
\end{longtable}
\ni The relations HX and HY can be condensed to $[h_b,x_a] = \theta_{ab} x_a$ and $[h_b,y_a] = -\theta_{ab} y_a$ for $a,b \in \Gamma$.  The algebra $\mf{h}'$ is the Lie algebra generated by $\{h_a\}_{a \in \Gamma}$ subject to the relation HH, and so it is abelian.  The algebra $\mf{n}_+$ (respectively $\mf{n}_-$) is the Lie algebra generated by $\{x_a\}_{a \in \Gamma}$ (respectively $\{y_a\}_{a \in \Gamma}$) subject to the relations XX (respectively YY).  The algebra $\mf{b}'_+$ (respectively $\mf{b}'_-$) is the Lie algebra generated by $\{x_a,h_a\}_{a \in \Gamma}$ (respectively $\{y_a,h_a\}_{a \in \Gamma}$) subject to the relations XX, HH, and HX (respectively YY, HH, and HY).  We note that $\mf{b}'_+ = \mf{h}' + \mf{n}_+$ and $\mf{b}'_- = \mf{h}' + \mf{n}_-$.  The algebra $\mf{g}'$ is the Lie algebra generated by $\{x_a,y_a,h_a\}_{a \in \Gamma}$ subject to all of the relations above.  When $\mf{g}$ is finite dimensional (and consequently semisimple), the algebras $\mf{h}$, $\mf{b}_\pm$, and $\mf{g}$ are respectively equal to $\mathfrak{h}', \mathfrak{b}'_\pm$, and $\mathfrak{g}'$.  

Let $\mc{V}$ be any vector space.  We say an operator $T : \mc{V} \to \mc{V}$ is \emph{square nilpotent} if $T^2 = 0$.  Consider actions on $\mc{V}$ of the generators $x_a$, $y_a$, and $h_a$ that are respectively given by operators $X_a$, $Y_a$, and $H_a$ in $\t{End}(\mc{V})$ for all $a \in \Gamma$.  If the $X_a$ (respectively $Y_a$) are square nilpotent for all $a \in \Gamma$, we say their actions are collectively \emph{$X$-square} (respectively \emph{$Y$-square}) \emph{nilpotent}.  An \emph{$\mf{h}'$-weight basis} of $\mc{V}$ is a basis $\mc{B}$ of $\mc{V}$ that simultaneously diagonalizes the operators $\{H_a\}_{a \in \Gamma}$.  A \emph{weight function} on $\mc{B}$ is a $\Gamma$-set of $\mb{C}$-valued functions on $\mc{B}$. Here the \emph{$\mf{h}'$-weight} of $\{H_a\}_{a \in \Gamma}$ is the weight function $\{\xi_a\}_{a \in \Gamma}$ satisfying $H_a.v = \xi_a(v).v$ for every $a \in \Gamma$ and every $v \in \mc{B}$.  The \emph{eigenvalue set} for $\mc{B}$ is $\mc{E}_{\mf{h}'} := \{\xi_a(v) \ | \ a \in \Gamma, v \in \mc{B}\}$.

\section{Representations of Lie algebras built from colored posets} \label{284}

Our standard context consists of the fixed locally finite poset $P$ that is colored by $\kappa$ with the fixed finite simple graph $\Gamma$, and that has a lattice of splits $\mc{FI}(P)$.  All of the results in this paper are stated in this context.  Let $\langle \mathcal{FI}(P) \rangle$ be the free complex vector space on $\mc{FI}(P)$.  In this section set $V := \langle \mathcal{FI}(P) \rangle$.  For each split $(F,I)$, denote the corresponding vector in $V$ by $\la F,I \ra$.  
We now define our raising and lowering operators on $V$.  Whenever we create these operators, we will be assuming the property EC which is defined below in Proposition \ref{323}.  This property implies that the defining sums for these operators will be either a single term or zero (and hence finite).  See Remark \ref{ECcomment}.

\begin{Def}\label{RaisingLoweringOperators}
Let $a \in \Gamma$.  For every $(F,I) \in \mathcal{FI}(P)$, define $X_a.\la F,I \ra := \sum \la F-x, I+x \ra$; here the sum is taken over all elements $x$ of color $a$ that are minimal in $F$.  Dually, define $Y_a.\la F,I \ra := \sum \la F+y,I-y \ra$; here the sum is taken over all elements $y$ of color $a$ maximal in $I$.  Linearly extending $X_a$ and $Y_a$ to all of $V$ respectively gives the \emph{color raising operator} and \emph{color lowering operator} for the color $a$.
\end{Def} 
\ni For $a \in \Gamma$, the action of $X_a$ (respectively $Y_a$) on a basis vector $\la F,I \ra$ can be viewed as summing over all ways to move up (respectively down) in $\mc{FI}(P)$ from $(F,I)$ by an edge colored $a$.  We extend our standard context to include $V$ (usually implicitly) and the operators $\{X_a,Y_a\}_{a \in \Gamma}$.  We use these operators to build representations from colored posets.

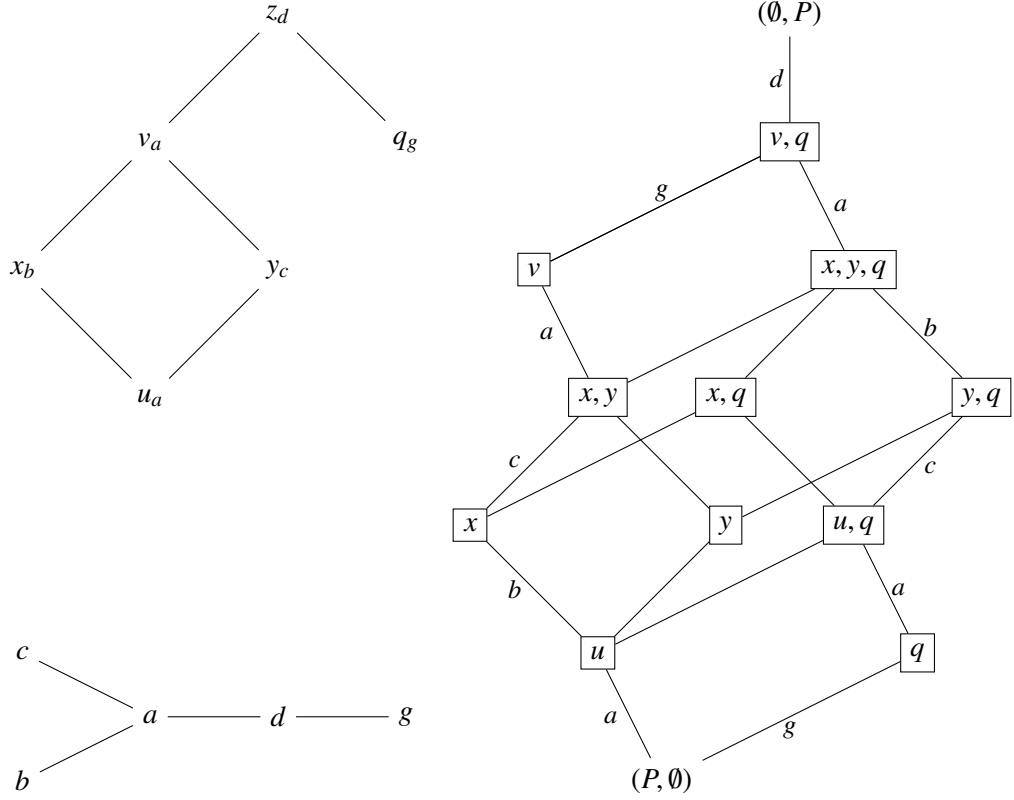
\begin{figure}[t!]
    \centering
    \begin{tikzpicture}[scale=1.7]
        \node(b) at (0,0){$b$};
        \node(c) at (0,1){$c$};
        \node(a) at (1,.5){$a$};
        \node(d) at (2,.5){$d$};
        \node(g) at (3,.5){$g$};
        
        \draw (b) -- (a) -- (c);
        \draw (a) -- (d) -- (g);
        
        \node(x) at (0,4){$x_b$};
        \node(u) at (1,3){$u_a$};
        \node(v) at (1,5){$v_a$};
        \node(y) at (2,4){$y_c$};
        \node(z) at (2,6){$z_d$};
        \node(q) at (3,5){$q_g$};
        
        \draw (q) -- (z) -- (v) -- (x) -- (u) -- (u) -- (y) -- (v);
        
        \node(P) at (5,0){$(P,\emptyset)$};
        \node(U)[draw] at (4.5,1){$u$};
        \node(X)[draw] at (3.5,2){$x$};
        \node(Y)[draw] at (5.5,2){$y$};
        \node(XY)[draw] at (4.5,3){$x,y$};
        \node(V)[draw] at (4,4){$v$};
        \node(Q)[draw] at (7,1){$q$};
        \node(QU)[draw] at (6.5,2){$u,q$};
        \node(QX)[draw] at (5.5,3){$x,q$};
        \node(QY)[draw] at (7.5,3){$y,q$};
        \node(QXY)[draw] at (6.5,4){$x,y,q$};
        \node(QV)[draw] at (6,5){$v,q$};
        \node(E) at (6,6){$(\emptyset,P)$};
        
        \draw (E) -- (QV) -- (V) -- (XY) -- (X) -- (U) -- (P) -- (Q) -- (QU) -- (QY) -- (QXY) -- (QV);
        \draw (U) -- (Y) -- (XY);
        \draw (QU) -- (QX) -- (QXY);
        \draw (U) -- (QU);
        \draw (Y) -- (QY);
        \draw (X) -- (QX);
        \draw (XY) -- (QXY);
        \draw (V) -- (QV);
        
        \node (A1) at (4.6,.5){\small $a$};
        \node (B1) at (3.85,1.5){\small $b$};
        \node (C1) at (3.85,2.5){\small $c$};
        \node (A2) at (4.1,3.5){\small $a$};
        \node (G1) at (5,4.6){\small $g$};
        \node (D1) at (5.9,5.5){\small $d$};
        \node (A3) at (6.4,4.5){\small $a$};
        \node (B2) at (7.1,3.55){\small $b$};
        \node (C2) at (7.1,2.45){\small $c$};
        \node (A4) at (6.85,1.5){\small $a$};
        \node (G2) at (6,.4){\small $g$};
    \end{tikzpicture}
\vs
    \caption{Clockwise from bottom left: Simple graph $\Gamma$ of finite type $D_5$, poset $P$ colored by $\Gamma$, and its edge-colored lattice $\mc{FI}(P)$ of filter-ideal splits.  Subscripts in $P$ indicate element colors.  Splits $(F,I)$ with boxes are described by the maximal elements of the ideal $I$.  Parallel edges in $\mc{FI}(P)$ have the same color.}
    \label{289}
\end{figure}

\begin{Def}\label{Carries}
Let $\mathfrak{L}$ be one of the algebras $\mathfrak{n}_+$, $\mathfrak{n}_-$, $\mathfrak{b}'_+$, $\mathfrak{b}'_-$, or $\mathfrak{g}'$.  We say that $\mathcal{FI}(P)$ \emph{carries a representation of $\mathfrak{L}$} if there is a linear operator on $V$ for each generator of $\mathfrak{L}$ specified in Section \ref{251} such that: 
\begin{enumerate}[(i),nosep]
\item For all $a \in \Gamma$: Whenever $x_a$ (respectively $y_a$) is in $\mathfrak{L}$, its operator is $X_a$ (respectively $Y_a$).
\item For all $a \in \Gamma$: Whenever $h_a$ is in $\mathfrak{L}$, its operator $H_a$ is diagonal with respect to the basis $\{\la F,I \ra\}_{(F,I) \in \mathcal{FI}(P)}$.
\item This collection of operators satisfies the defining relations for $\mathfrak{L}$ given in Section \ref{251} with respect to the commutator $[A,B] := AB - BA$ on $\t{End}(V)$.
\end{enumerate} 
\end{Def}

\ni Let $\mathfrak{L}$ be one of the algebras $\mathfrak{n}_+$, $\mathfrak{n}_-$, $\mathfrak{b}'_+$, $\mathfrak{b}'_-$, or $\mathfrak{g}'$.  The relations from Section \ref{251} required by (iii) above make $V$ into a module for $\mathfrak{L}$ with the required operators giving the actions of the generators of $\mathfrak{L}$.  This module structure induces a Lie algebra homomorphism from $\mathfrak{L}$ to $\mathfrak{gl}(V)$, and so $V$ is a representation of $\mathfrak{L}$.  

The lattice of splits in Figure \ref{289} carries a representation of $\mf{b}'_+$, where $\mf{g}$ is the algebra of finite type $D_5$; this illustrates Theorem \ref{388}.  The lattice of splits in Figure \ref{134} carries a representation of $\mf{g}'$, where $\mf{g}$ is the algebra of affine type $\tilde{E}_6$; this illustrates Theorem \ref{653}.  The minuscule representations built from colored minuscule posets and the representations built from full heaps mentioned in Section \ref{119} are representations of $\mf{g}'$ carried by $\mc{FI}(P)$.  

\begin{Rmk}
\begin{enumerate}[(a),nosep]
\item The operators $\{X_a\}_{a \in \Gamma}$ and $\{Y_a\}_{a \in \Gamma}$ are determined uniquely by the colored structure of $P$, so any representation of $\mathfrak{n}_+$ or $\mathfrak{n}_-$ carried by $\mathcal{FI}(P)$ is necessarily unique.  
\item As we work in the context of Definition \ref{Carries}, we will leave the actions of $x_a, y_a$, and $h_a$ implicit and will refer only to the operators $X_a, Y_a$, and $H_a$.  We will also routinely leave (i) and (ii) implicit when showing $\mathcal{FI}(P)$ carries a representation of some algebra.  Our focus will be (iii); that is, showing the operators satisfy the relations for the algebra.
\item Any diagonal operators $\{H_a\}_{a \in \Gamma}$ satisfying (ii) automatically satisfy the relation HH.  Additionally, these operators make $\{\la F,I \ra\}_{(F,I) \in \mathcal{FI}(P)}$ an $\mathfrak{h}'$-weight basis of $V$.
\end{enumerate}
\end{Rmk}


We now present the central representation definitions of this paper.  
\begin{Def}\label{MinReps} 
We define the ``minuscule'' representations of $\mathfrak{b}'_\pm$ and $\mathfrak{g}'$ built from colored posets:
\begin{enumerate}[(a),nosep]
\item We say a representation of $\mf{b}'_+$ (respectively $\mf{b}'_-$) carried by $\mc{FI}(P)$ is \emph{upper} (\emph{lower}) \emph{$P$-minuscule} if it is $X$-square ($Y$-square) nilpotent and the actions of the diagonal operators $\{H_a\}_{a \in \Gamma}$ satisfy:
\be[(i),nosep] 
\item The set $\mc{E}_{\mf{h}'}$ of eigenvalues of the $\{H_a\}_{a \in \Gamma}$ is contained in $\{-1,0,1,2,\dots\}$ (respectively contained in $\{\dots,-2,-1,0,1\}$), 
\item For any split $(F,I)$ and any $a \in \Gamma$ we have $H_a.\la F,I \ra = -\la F,I \ra$ (respectively $H_a.\la F,I \ra = +\la F,I \ra$) if and only if $F$ (respectively $I$) has a minimal (maximal) element of color $a$.
\ee
\item We say a representation of $\mf{g}'$ carried by $\mc{FI}(P)$ is \emph{$P$-minuscule} if the set $\mc{E}_{\mf{h}'}$ of eigenvalues of the $\{H_a\}_{a \in \Gamma}$ is contained in $\{-1,0,1\}$.
\end{enumerate}
\end{Def}
\ni The lattice of splits in Figure \ref{289} carries an upper $P$-minuscule representation of $\mathfrak{b}'_+$; this illustrates Theorem \ref{614}.  For this representation, we have $H_a.\la P,\emptyset \ra = - \la P, \emptyset \ra$, $H_g.\la P,\emptyset \ra = -\la P, \emptyset \ra$, $H_b.\la P,\emptyset \ra = 0$, $H_c.\la P,\emptyset \ra = 0$, and $H_d.\la P,\emptyset \ra = +2\la P, \emptyset \ra$.  The actions of the $\{H_e\}_{e \in \Gamma}$ at other splits can be computed by working up through $\mc{FI}(P)$ using the relations HX, so one can calculate that $\mc{E}_{\mf{h}'} = \{-1,0,1,2\}$.  

The $P$-minuscule representations of $\mf{g}'$ are $X$- and $Y$-square nilpotent, as is implied by Lemma \ref{619}.  The lattice $\mc{FI}(P)$ of splits displayed in Figure \ref{134} carries a $P$-minuscule representation of $\mf{g}'$; this illustrates Theorem \ref{653}.



\section{Square nilpotent representations of $\boldsymbol{\mf{n}_+}$ and $\boldsymbol{\mf{n}_-}$}\label{317}

We establish our earliest equivalences between sets of coloring properties and sets of algebraic conditions.  Theorem \ref{351} summarizes this section by listing three coloring properties for $P$ that are necessary and sufficient for the operators $\{X_a\}_{a \in \Gamma}$ to generate an $X$-square nilpotent representation of $\mf{n}_+$ carried by $\mc{FI}(P)$.

\begin{Prop}\label{323}
The following are equivalent:
\be[(i),nosep]
\item The color raising operators $\{X_a\}_{a \in \Gamma}$ are $X$-square nilpotent.
\item The following two properties are satisfied by $P$:
\bi[nosep]
\i (EC): Elements with equal colors are comparable, and
\i (ND): Neighbors have different colors.
\ei
\item The color lowering operators $\{Y_a\}_{a \in \Gamma}$ are $Y$-square nilpotent.
\ee
\end{Prop}

\begin{Rmk} \label{ECcomment}
For all $a \in \Gamma$ and every split $(F,I)$, the property EC implies that the sum defining $X_a.\la F,I \ra$ (or $Y_a.\la F,I \ra$) is either a single term or is zero since there may be at most one minimal (or maximal) element of $F$ (or $I$) of a given color.  This remark will often be used implicitly.  From now on when we are creating the color raising and lowering operators we will be assuming the property EC holds.  The properties EC and ND together ensure that no two edges of the same color are incident in the Hasse diagram of $\mc{FI}(P)$.
\end{Rmk}

\bp[Proof of Proposition \ref{323}.]
We show that (i) and (ii) are equivalent; the equivalence of (ii) and (iii) follows from a dualized argument.  To prove (i) implies (ii), first suppose that EC fails.  Then there is a color $a \in \Gamma$ and incomparable elements $x,y \in P_a = \kappa^{-1}(a)$.  Let $F$ be the filter generated by $x$ and $y$.  Since $2 \la F-x-y, I+x+y \ra$ is a term in the expansion of $X_a^2.\la F,I \ra$, we have $X_a^2.\la F,I \ra \ne 0$.  Now suppose ND fails.  Then there is a color $a \in \Gamma$ and neighbors $x \to y$, with $x,y \in P_a$.  Let $F$ be the filter generated by $x$.  Since $\la F-x-y, I+x+y \ra$ is a term in the expansion of $X_a^2.\la F,I \ra$, we have $X_a^2.\la F,I \ra \ne 0$.  Thus (ii) holds.

Now suppose (ii) holds and fix $a \in \Gamma$.  Let $(F,I) \in \mc{FI}(P)$.  Suppose $X_a.\la F,I \ra = \la F-x,I+x \ra$ for some element $x \in P_a$.  Let $y \in P_a$; here $y$ is comparable to $x$ by EC.  However, note $y$ cannot cover $x$ by ND.  Hence $y$ is not minimal in $F-x$.  Since $y \in P_a$ was arbitrary, we have $X_a.\la F-x,I+x \ra = 0$.  So for every color $a \in \Gamma$ and $(F,I) \in \mc{FI}(P)$ we have $X_a^2.\la F,I \ra = 0$.  Thus (i) holds.
\ep

\ni We get the relations XX(i) and YY(i) by strengthening ND.

\begin{Lem}\label{341}
Suppose $P$ satisfies EC and ND.  Then the following are equivalent:
\be[(i),nosep]
\item The relation $[X_b,X_a]=0$ holds if $a,b \in \Gamma$ are distant.
\item The following additional property is satisfied by $P$:
\bi[nosep]
\i (NA): Neighbors have adjacent colors.
\ei
\item The relation $[Y_b,Y_a]=0$ holds if $a,b \in \Gamma$ are distant.
\ee
\end{Lem}

\bp
\ni We show (i) and (ii) are equivalent; the equivalence of (ii) and (iii) follows from a dualized argument. To prove (i) implies (ii), first suppose (ii) fails.  Then there exist neighbors $x \to y$ such that either $\kappa(x) = \kappa(y)$ or $\kappa(x) \dist \kappa(y)$.  By ND we know $\kappa(x) \ne \kappa(y)$, and so $\kappa(x) \dist \kappa(y)$.  Let $a := \kappa(x)$ and $b := \kappa(y)$, and let $F$ be the filter generated by $x$.  Since $\la F-x-y, I+x+y \ra$ is a term in the expansion of $X_bX_a.\la F,I \ra$, we have $X_bX_a.\la F,I \ra \ne 0$.  However, we get $X_aX_b.\la F,I \ra = 0$ since the only minimal element in $F$ has color $a$.  Hence $[X_b,X_a].\la F,I \ra \ne 0$, and so (i) fails.

To prove (ii) implies (i), suppose (ii) holds and let $a$ and $b$ be distant colors.  Let $(F,I)$ be any split.  Note that if $X_aX_b.\la F,I \ra = 0$ and $X_bX_a.\la F,I \ra = 0$, then we are done.  Next, without loss of generality assume $X_bX_a.\la F,I \ra \ne 0$.  Then there are elements $x$ and $y$ such that $\kappa(x) = a$ and $\kappa(y) = b$ and $X_bX_a.\la F,I \ra = \la F-x-y,I+x+y \ra$.  By NA we see $x$ and $y$ are not neighbors.  Thus they are incomparable minimal elements of $F$.  Hence $X_aX_b.\la F,I \ra = \la F-y-x,I+y+x \ra = X_bX_a. \la F,I \ra$, so $[X_b,X_a].\la F,I \ra = 0$.  Thus (i) holds.
\ep

We get the relations XX(ii) and YY(ii) by introducing a special case of the future key property I2A.

\begin{Lem}\label{346}
Suppose $P$ satisfies EC and ND.  Then the following are equivalent:
\be[(i),nosep]
\item The relation $[X_a,[X_a,X_b]] = 0$ holds for all $a,b \in \Gamma$.
\item The following additional property is satisfied by $P$:
\bi[nosep]
\i (I3ND): If three successive neighbors $x \to y \to z$ form an interval in $P$, then $x$ and $z$ have different colors.
\ei
\item The relation $[Y_a,[Y_a,Y_b]]=0$ holds for all $a,b \in \Gamma$.
\ee
\end{Lem}

\bp We show (i) and (ii) are equivalent; the equivalence of (ii) and (iii) follows from a dualized argument.  For all $a,b \in \Gamma$, note $[X_a,[X_a,X_b]] = X_a^2 X_b - 2X_aX_bX_a + X_bX_a^2$.  By Proposition \ref{323} the first and last terms vanish when acting on any split.  Thus the relation $[X_a,[X_a,X_b]] = 0$ holds if and only if $X_aX_bX_a = 0$ holds. 

Suppose (i) holds, so that for all $a,b \in \Gamma$ we have $X_aX_bX_a = 0$.  Suppose three successive neighbors $x \to y \to z$ form an interval in $P$.  Define $a := \kappa(x)$ and $b := \kappa(y)$ and $c := \kappa(z)$.  Let $F$ be the filter generated by $x$.  Note that $X_cX_bX_a. \la F,I \ra = \la F-x-y-z,I+x+y+z \ra \ne 0$.  Hence we have $c \ne a$, so I3ND holds.

Now suppose (ii) holds and let $a,b \in \Gamma$.  Assume for a contradiction that $X_aX_bX_a \ne 0$.  Then there is a split $(F,I)$ and elements $x,y,z$ with $\kappa(x) = \kappa(z) = a$ and $\kappa(y) = b$ such that $X_aX_bX_a.\la F,I \ra = \la F-x-y-z,I+x+y+z \ra$.  By EC we know that $x \le z$, and by ND we know that $z$ cannot cover $x$.  Hence the open interval $(x,z)$ is nonempty.  Since $z$ is minimal in $F-x-y$, we see that $(x,z) = \{y\}$.  Thus $x \to y \to z$ is an interval of three successive neighbors in $P$.  This violates I3ND.  Hence $X_aX_bX_a = 0$, and so (i) holds.
\ep

Since NA implies ND, we can combine the three results above to produce

\begin{Thm}\label{351}
The following are equivalent:
\be[(i),nosep]
\item The lattice $\mc{FI}(P)$ carries an $X$-square nilpotent representation of $\mf{n}_+$.
\item The properties EC, NA, and I3ND are satisfied by $P$.
\item The lattice $\mc{FI}(P)$ carries a $Y$-square nilpotent representation of $\mf{n}_-$.
\ee
\end{Thm}


We remark that if $a,b \in \Gamma$ are adjacent, then $[X_b,X_a] \ne 0$ if and only if there exist neighbors in $P$ with those two colors.  This will not be used here; see \cite[Prop. 3.1.5]{Str} for details.

\section{Square nilpotent representations of $\boldsymbol{\mf{b}'_+}$ and $\boldsymbol{\mf{b}'_-}$}\label{361}

Here we define two kinds weight functions on $\mathcal{FI}(P)$ and show that they are equivalent in Proposition \ref{377}.  In Proposition \ref{378}, we show that the diagonal operators on $\langle \mathcal{FI}(P) \rangle$ corresponding to the first kind satisfy the relations HX and HY.  We use these operators in Theorem \ref{388} to extend the representations of Theorem \ref{351} from $\mathfrak{n}_+$ (respectively $\mathfrak{n}_-$) to $\mathfrak{b}'_+$ (respectively $\mathfrak{b}'_-$).



Our first kind of weight function is defined with a local $\mathfrak{h}'$-weight comparison along each edge of $\mathcal{FI}(P)$.
\begin{Def}\label{EdgeWeight}
Let $\{\eta_a\}_{a \in \Gamma}$ be a weight function on $\mathcal{FI}(P)$.  We call $\{\eta_a\}_{a \in \Gamma}$ an \emph{edge weight function} if for every $b \in \Gamma$, split $(F,I) \in \mathcal{FI}(P)$, and minimal element $x \in F$, we have
\begin{align} \label{373} 
\eta_b(F-x,I+x) - \eta_b(F,I) = \theta_{\kappa(x),b}.
\end{align}
\end{Def}



Our second kind of weight function compares
the $\mf{h}'$-weights of two splits at a distance.  
For this definition, first partition
$\mc{FI}(P)$ into \emph{components} as follows: Two splits $(F,I)$ and $(F',I')$ are in the same component if there is a path from $(F,I)$ to $(F',I')$ in the Hasse diagram of $\mc{FI}(P)$ consisting of finitely many edges.  Such a path can progress both upward and downward along edges.  
An example of a poset $P$ for which $\mc{FI}(P)$ has more than one component is $P = \mb{Z}$.  Here there will be three components: The component $\{(\emptyset,\mb{Z})\}$, the component $\{(\mb{Z},\emptyset)\}$, and the component containing every other split.  The finiteness of $\Gamma$ can be used \cite[Prop. 3.2.2]{Str} to show that $\mathcal{FI}(P)$ has finitely many components when $P$ satisfies EC.  We also have

\begin{Lem} \label{366}
When $P$ is finite, there is only one component of $\mathcal{FI}(P)$.
\end{Lem}

\begin{proof}
Let $(F,I)$ and $(F',I')$ be splits.  Since $P$ is finite, so are the sets $I - I'$ and $I' - I$.  A finite path from $(F,I)$ to $(F',I')$ can always be constructed by first transferring all of the elements of $I - I'$ to filters and then transferring all of the elements of $I' - I$ to ideals.
%
\end{proof}


Suppose $(F,I)$ and $(F',I')$ are in the same component of $\mc{FI}(P)$.  The edges in any path from $(F,I)$ to $(F',I')$ have the net effect of transferring every element of $I' - I$ to ideals and transferring every element of $I - I'$ to filters.  Fix $b \in \Gamma$ and note that both $P_b \cap (I' - I)$ and $P_b \cap (I - I')$ are finite.  So in any path from $(F,I)$ to $(F',I')$, the cardinality $|P_b \cap (I' - I)|$ counts the net number of edges of color $b$ traversed upward and $|P_b \cap (I - I')|$ counts the net number of edges of color $b$ traversed downward.  

\begin{Def} 
Let $b \in \Gamma$.  For any splits $(F,I)$ and $(F',I')$ that are in the same component of $\mathcal{FI}(P)$, set $\Delta_b[(F',I'),(F,I)] := |P_b \cap (I' - I)| - |P_b \cap (I - I')|$.
\end{Def}

\ni This is the signed net number of edges of color $b$ traversed in any finite path from $(F,I)$ to $(F',I')$.  
Note that $\Delta_b[(F',I'),(F,I)] = -\Delta_b[(F,I),(F',I')]$.  

We now define our second kind of weight function; Remark \ref{369} motivates the right hand side of (\ref{371}).

\begin{Def}\label{ComponentWeightFunction}
Let $\{\eta_a\}_{a \in \Gamma}$ be a weight function on $\mc{FI}(P)$.  We call $\{\eta_a\}_{a \in \Gamma}$ a \emph{component weight function} if for every $b \in \Gamma$, whenever $(F,I)$ and $(F',I')$ are in the same component of $\mc{FI}(P)$ we have
\begin{align} \label{371}
\eta_b(F',I') - \eta_b(F,I) = 2\Delta_b[(F',I'),(F,I)] - \sum_{c \adj b} \Delta_c[(F',I'),(F,I)]. 
\end{align}
\end{Def}


\begin{Rmk}\label{369}
\begin{enumerate}[(a),nosep]
\item Let $(F,I) \in \mathcal{FI}(P)$ and suppose $x$ is minimal in $F$.  Note that $\Delta_d[(F-x,I+x),(F,I)] = \delta_{\kappa(x),d}$ for all $d \in \Gamma$.  Thus we have $2\Delta_b[(F-x,I+x),(F,I)] - \sum_{c \sim b} \Delta_c[(F-x,I+x),(F,I)] = \theta_{\kappa(x),b}$ for every $b \in \Gamma$, since $2\delta_{\kappa(x),b} - \sum_{c \adj b} \delta_{\kappa(x),c} = \theta_{\kappa(x),b}$.
\item We see from (a) that Equation (\ref{371}) specializes to Equation (\ref{373}) when $(F,I) \to (F',I')$ in $\mathcal{FI}(P)$, since then $(F',I') = (F-x,I+x)$ for some minimal $x \in F$.
\end{enumerate}
\end{Rmk}

The following lemma is used in Proposition \ref{377} to show Definitions \ref{EdgeWeight} and \ref{ComponentWeightFunction} are equivalent.





\begin{Lem} \label{376}
Let $b \in \Gamma$ and let $(F,I)$, $(F',I')$, and $(F'',I'')$ be splits in the same component of $\mc{FI}(P)$.
\be[(a),nosep]
\item We have $\Delta_b[(F'',I''),(F,I)] = \Delta_b[(F'',I''),(F',I')] + \Delta_b[(F',I'),(F,I)]$.
\item Let $\{\eta_a\}_{a \in \Gamma}$ be a weight function on $\mc{FI}(P)$.  If Equation (\ref{371}) holds for the pairs of splits $(F,I),(F',I')$ and $(F',I'),(F'',I'')$, then it also holds for the pair $(F,I),(F'',I'')$.
\ee
\end{Lem}

\begin{proof}
Let $\mathcal{P}_1$ be a finite path in $\mathcal{FI}(P)$ from $(F,I)$ to $(F',I')$, and let $\mathcal{P}_2$ be a finite path in $\mathcal{FI}(P)$ from $(F',I')$ to $(F'',I'')$.  Let $\mathcal{P}$ be the concatenation of $\mathcal{P}_1$ followed by $\mathcal{P}_2$.  The quantities $\Delta_b[(F',I'),(F,I)]$ and $\Delta_b[(F'',I''),(F',I')]$ are the respective signed net number of edges of color $b$ traversed along $\mathcal{P}_1$ and $\mathcal{P}_2$.  Thus their sum is the signed net number of edges of color $b$ traversed along $\mathcal{P}$.  But this quantity is also $\Delta_b[(F'',I''),(F,I)]$, so we get (a).  Then (b) follows from (a) by noting that $\eta_b(F'',I'') - \eta_b(F,I) = \eta_b(F'',I'') - \eta_b(F',I') + \eta_b(F',I') - \eta_b(F,I)$.
\end{proof}



\begin{Prop}\label{377}
Let $\{\eta_a\}_{a \in \Gamma}$ be a weight function on $\mathcal{FI}(P)$.
Then $\{\eta_a\}_{a \in \Gamma}$ is an edge weight function if and only if it is a component weight function.
\end{Prop}

\begin{proof}
First suppose $\{\eta_a\}_{a \in \Gamma}$ is a component weight function.  Let $b \in \Gamma$ and $(F,I) \in \mathcal{FI}(P)$ and suppose $x$ is minimal in $F$.  Using Equation (\ref{371}) with $(F',I') = (F-x,I+x)$ and then applying Remark \ref{369}(a) produces Equation (\ref{373}).  Hence $\{\eta_a\}_{a \in \Gamma}$ is an edge weight function.  

Now suppose $\{\eta_a\}_{a \in \Gamma}$ is an edge weight function.  Let $b \in \Gamma$ and suppose $(F,I)$ and $(F',I')$ are in the same component of $\mathcal{FI}(P)$.  
If $(F,I) = (F',I')$, then both sides of Equation (\ref{371}) vanish.
Next
suppose $(F,I)$ and $(F',I')$ are neighbors in $\mathcal{FI}(P)$.  If $(F,I) \to (F',I')$, then $(F',I') = (F-x,I+x)$ for some $x$ minimal in $F$.  Applying Equation (\ref{373}) followed by Remark \ref{369}(a) produces Equation (\ref{371}).  If $(F',I') \to (F,I)$, then a similar argument gives a variant of (\ref{371}) with the roles of $(F,I)$ and $(F',I')$ reversed.  Multiplying by $-1$ then produces (\ref{371}) since $-\Delta_d[(F,I),(F',I')] = \Delta_d[(F',I'),(F,I)]$ for all $d \in \Gamma$.  
Thus (\ref{371}) holds for all pairs of neighbors in $\mathcal{FI}(P)$.
In general, choose a path of length $\ell$ from $(F,I)$ to $(F',I')$ in $\mathcal{FI}(P)$.
Since Equation (\ref{371}) holds for all pairs of neighbors in $\mathcal{FI}(P)$, we may apply Lemma \ref{376}(b) a total of $\ell - 1$ times along 
this path
to produce Equation (\ref{371}) for $(F,I)$ and $(F',I')$.  Thus $\{\eta_a\}_{a \in \Gamma}$ is a component weight function.
\end{proof}

Now we show a component weight function $\{\eta_a\}_{a \in \Gamma}$ always exists.  Fix $b \in \Gamma$, a component $\mc{C}$, and any split $(F_0,I_0) \in \mc{C}$.  Let $\eta_b(F_0,I_0)$ be any complex number.  Then for $(F,I) \in \mc{C}$, define
\begin{align}\label{375} 
\eta_b(F,I) := \eta_b(F_0,I_0) + 2\Delta_b[(F,I),(F_0,I_0)] - \sum_{c \adj b} \Delta_c[(F,I),(F_0,I_0)].
\end{align}
\ni For each $b \in \Gamma$, make such choices for all components and then perform this construction.

\begin{Lem}\label{382} (a) The $\{\eta_a\}_{a \in \Gamma}$ defined by (\ref{375}) is a component weight function. 
\vs
\be[(a),nosep]
\setcounter{enumi}{1}
\vspace{-.1in}
\item Each $\eta_a$ in some component weight function $\{\eta_a\}_{a \in \Gamma}$ is uniquely determined by its value on one split for each component. \ee
\end{Lem} 

\bp
Fix $b \in \Gamma$, a component $\mc{C}$, and the split $(F_0,I_0)$ chosen for $\mc{C}$.  Let $(F,I),(F',I') \in \mc{C}$.  Using (\ref{375}) we know Equation (\ref{371}) holds for the pairs $(F,I),(F_0,I_0)$ and $(F',I'),(F_0,I_0)$.  Since (\ref{371}) holds also for the pair $(F_0,I_0),(F',I')$, Lemma \ref{376}(b) shows it holds for $(F,I),(F',I')$.  Thus we get (a).  We get (b) since each $\eta_a$ in a given $\{\eta_a\}_{a \in \Gamma}$ must satisfy (\ref{375}) once its values are specified on one split from each component.  
\ep



We name the diagonal operators whose $\mathfrak{h}'$-weights are component (or edge) weight functions.

\begin{Def}
We say that a $\Gamma$-set of diagonal operators $\{H_a\}_{a \in \Gamma}$ on $\langle \mathcal{FI}(P) \rangle$ are \emph{component} (respectively \emph{edge}) \emph{diagonal operators} if their $\mathfrak{h}'$-weight $\{\eta_a\}_{a \in \Gamma}$ is a component (respectively edge) weight function.
\end{Def}

\ni These are the diagonal operators needed to satisfy HX (or HY) when $P$ satisfies the property EC.  
We will typically use edge weight functions and edge diagonal operators and apply Proposition \ref{377} as needed.




\begin{Prop} \label{378}
Suppose $P$ satisfies EC.  Let $\{H_a\}_{a \in \Gamma}$ be diagonal operators with $\mf{h}'$-weight $\{\eta_a\}_{a \in \Gamma}$.  Then the following are equivalent:
\be[(i),nosep]
\item The operators $\{X_a,H_a\}_{a \in \Gamma}$ satisfy HX.
\item The operators $\{H_a\}_{a \in \Gamma}$ are edge diagonal operators.
\item The operators $\{Y_a,H_a\}_{a \in \Gamma}$ satisfy HY. 
\ee
\end{Prop}

\bp
Suppose (i) holds.  Let $b \in \Gamma$ and $(F,I) \in \mathcal{FI}(P)$, and suppose $x$ is a minimal element of $F$.  Set $a := \kappa(x)$.  By EC we have $X_a.\la F,I \ra = \la F-x,I+x \ra$.  Then HX gives $\theta_{ab}X_a.\la F,I \ra = (H_bX_a - X_aH_b).\la F,I \ra = (\eta_b(F-x,I+x) - \eta_b(F,I))X_a.\la F,I \ra$, so $\eta_b(F-x,I+x) - \eta_b(F,I) = \theta_{ab}$.  Thus $\{\eta_a\}_{a \in \Gamma}$ is an edge weight function, and so (ii) holds.

Now suppose (ii) holds and let $a,b \in \Gamma$ and $(F,I) \in \mathcal{FI}(P)$.  
Note that $(H_bX_a - X_aH_b).\la F,I \ra = \theta_{ab}X_a. \la F,I \ra$ is trivial if $F$ does not have a minimal element of color $a$, so suppose $x$ is minimal in $F$ of color $a$.  By EC we have $X_a.\la F,I \ra = \la F-x,I+x \ra$.  Then $(H_bX_a - X_aH_b).\la F,I \ra = (\eta_b(F-x,I+x) - \eta_b(F,I))X_a.\la F,I \ra = \theta_{ab} X_a. \la F,I \ra$, where the last equality follows from 
Equation (\ref{373}).  Thus (i) holds.

The equivalence of (ii) and (iii) follows from a dualized argument.
\ep


We conclude this section by applying our results to representations.

\begin{Prop} \label{380}
Suppose $P$ satisfies EC.
\be[(a),nosep]
\item If $\mc{FI}(P)$ carries a representation of $\mf{b}'_+$ or of $\mf{b}'_-$, then the diagonal operators $\{H_a\}_{a \in \Gamma}$ giving the actions of $\{h_a\}_{a \in \Gamma}$ are edge diagonal operators.
\item Suppose $P$ additionally satisfies NA and I3ND.  If $\{\eta_a\}_{a \in \Gamma}$ is an edge weight function, then the corresponding edge diagonal operators $\{H_a\}_{a \in \Gamma}$ can be used to extend the representations of Theorem \ref{351} from $\mf{n}_+$ and $\mf{n}_-$ to $X$- and $Y$-square nilpotent representations of $\mf{b}'_+$ and $\mf{b}'_-$.
\ee 
\end{Prop}

\begin{proof}
For (a), the relations HX (or HY) are satisfied by assumption.  Thus (a) holds by Proposition \ref{378}.  Part (b) is also a consequence of Proposition \ref{378} since that result guarantees HX and HY hold.  
\end{proof}

Using the existence obtained in Lemma \ref{382}(a), we can extend Theorem \ref{351}.

\begin{Thm} \label{388}
The following are equivalent:
\be[(i),nosep]
\item The lattice $\mc{FI}(P)$ carries an $X$-square nilpotent representation of $\mf{n}_+$.
\item The lattice $\mc{FI}(P)$ carries an $X$-square nilpotent representation of $\mf{b}'_+$.
\item The properties EC, NA, and I3ND are satisfied by $P$.
\item The lattice $\mc{FI}(P)$ carries a $Y$-square nilpotent representation of $\mf{b}'_-$.
\item The lattice $\mc{FI}(P)$ carries a $Y$-square nilpotent representation of $\mf{n}_-$.
\ee
\ni Any operators $\{H_a\}_{a \in \Gamma}$ used to satisfy Part (ii) or Part (iv) are edge diagonal operators.
\end{Thm}

\bp 
Using Lemma \ref{382}(a), let $\{\eta_a\}_{a \in \Gamma}$ be any component weight function.  Then $\{\eta_a\}_{a \in \Gamma}$ is an edge weight function by Proposition \ref{377}.  Hence (iii) implies (ii) by Proposition \ref{380}(b) using $\{\eta_a\}_{a \in \Gamma}$.  Also (ii) implies (i) by restricting to the operators $\{X_a\}_{a \in \Gamma}$, and (i) implies (iii) by Theorem \ref{351}.  The equivalence of (iii), (iv), and (v) follows from a dualized argument.  The last statement follows from Proposition \ref{380}(a).
\ep 

\begin{Rmk}
If one does not care about the relationship between combinatorial properties and weights, Theorem \ref{388} says that a representation of $\mf{n}_+$ carried by $\mc{FI}(P)$ can be extended to $\mf{b}'_+$ without requiring coloring properties for $P$ beyond EC, NA, and I3ND.  One first creates a component weight function $\{\eta_a\}_{a \in \Gamma}$ by choosing for each component of $\mc{FI}(P)$ any $\Gamma$-set of complex numbers and any split.  Then $\{\eta_a\}_{a \in \Gamma}$ is also an edge weight function.  The corresponding edge diagonal operators $\{H_a\}_{a \in \Gamma}$ are then used to extend the action of $\mf{n}_+$.
\end{Rmk}

\section{A combinatorially constructed edge weight function}\label{414}

\ni We continue to assume $P$ satisfies EC.  Here we construct a particular weight function $\{\mu_a\}_{a \in \Gamma}$ on $\mc{FI}(P)$ whose values are determined by the local structure of $P$.  When $P$ has two new additional properties beyond EC, in Proposition \ref{423}(b) we show that $\{\mu_a\}_{a \in \Gamma}$ is an edge weight function.  As we work toward obtaining the (upper) $P$-minuscule representations of $\mf{g}'$ (and $\mf{b}'_+$), in the next two sections we will obtain relationships between further coloring properties and this $\mf{h}'$-weight.

We prepare to define our new $\mb{Z}$-valued weight function $\{\mu_a\}_{a \in \Gamma}$.  
To construct this weight function, we first introduce $\mathbb{N}$-valued auxiliary functions $\{\upsilon_a\}_{a \in \Gamma}$ and $\{\psi_a\}_{a \in \Gamma}$. 

\begin{Def} \label{Upsilon}
Fix $b \in \Gamma$.  We define $\upsilon_b : \mc{FI}(P) \to \mb{N}$ in stages.  Let $(F,I)$ be a split.  If $P_b \cap I$ does not have a maximal element, then set $\upsilon_b(F,I) := 1$.  Now suppose that $P_b \cap I$ has a maximal element $y$.  By EC the element $y$ is unique.  We build up a set $\Upsilon_b(F,I) \subseteq I$ from the empty set $\emptyset$.  Let $z \in I$.  We place $z$ into $\Upsilon_b(F,I)$ if it meets the following three requirements:
\be[(i),nosep]
\item The element $z$ is greater than $y$, 
\item Its color $c := \kappa(z)$ is adjacent to $b$, and
\item The number of elements greater than $y$ that are in $P_c \cap I$ is finite.
\ee
\ni If there is some color $a \adj b$ such that there are infinitely many elements greater than $y$ in $P_a \cap I$, then set $\upsilon_b(F,I) := |\Upsilon_b(F,I)| + 1$.  Otherwise set $\upsilon_b(F,I) := |\Upsilon_b(F,I)|$.
\end{Def}

Figure \ref{420} illustrates the three possible scenarios for computing $\{\upsilon_i\}_{i \in \Gamma}$ on a $\Gamma$-colored poset $P$ at a split $(F,I)$.  Since there is no maximal element of color $d$ in $I$, we have $\upsilon_d(F,I) = 1$.  Because $g \adj d$, we get $\upsilon_g(F,I) = |\Upsilon_g(F,I)| + 1 = 2$.  Lastly, we have $\upsilon_a(F,I) = |\Upsilon_a(F,I)| = 0$.


\begin{Def} \label{Psi}
Fix $b \in \Gamma$.  We define $\psi_b : \mc{FI}(P) \to \mb{N}$ in dually analogous stages.  Let $(F,I)$ be a split.  If $P_b \cap F$ does not have a minimal element, then set $\psi_b(F,I) := 1$.  Now suppose that $P_b \cap F$ has a minimal element $y$.  By EC the element $y$ is unique.  We build up another set $\Psi_b(F,I) \subseteq F$ from the empty set $\emptyset$.  Let $z \in F$.  We place $z$ into $\Psi_b(F,I)$ if it meets the following three requirements:
\be[(i),nosep]
\item The element $z$ is less than $y$, 
\item Its color $c := \kappa(z)$ is adjacent to $b$, and
\item The number of elements less than $y$ that are in $P_c \cap F$ is finite.
\ee
\ni If there is some color $a \adj b$ such that there are infinitely many elements less than $y$ in $P_a \cap F$, then set $\psi_b(F,I) := |\Psi_b(F,I)| + 1$.  Otherwise set $\psi_b(F,I) := |\Psi_b(F,I)|$.
\end{Def}

\ni Since $\Gamma$ is finite, Condition (iii) implies the sets $\Upsilon_b(F,I)$ and $\Psi_b(F,I)$ in the above definitions are finite.  

We now define the $\mathbb{Z}$-valued weight function $\{\mu_a\}_{a \in \Gamma}$ and its corresponding set of diagonal operators.

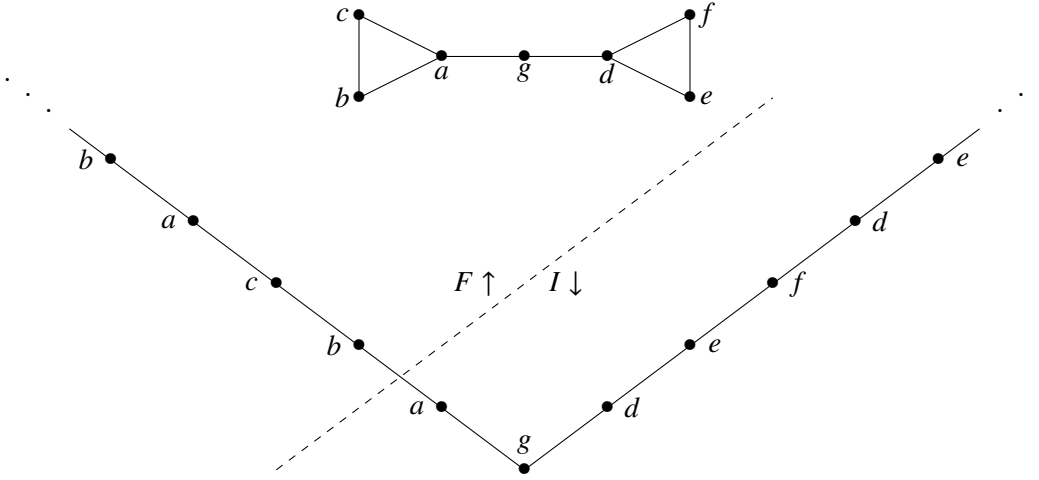
\begin{figure}[t!]
    \centering
    \begin{tikzpicture}[scale=1.1]
        \node at (0,0){$\bullet$};
        \node at (1,.75){$\bullet$};
        \node at (2,1.5){$\bullet$};
        \node at (3,2.25){$\bullet$};
        \node at (4,3){$\bullet$};
        \node at (5,3.75){$\bullet$};
        \node at (-1,.75){$\bullet$};
        \node at (-2,1.5){$\bullet$};
        \node at (-3,2.25){$\bullet$};
        \node at (-4,3){$\bullet$};
        \node at (-5,3.75){$\bullet$};
        
        \node at (5.75,4.3125){$\cdot$};
        \node at (6,4.5){$\cdot$};
        \node at (6.25,4.6875){$\cdot$};
        
         \node at (-5.75,4.3125){$\cdot$};
        \node at (-6,4.5){$\cdot$};
        \node at (-6.25,4.6875){$\cdot$};
        
        \draw (5.5,4.125) -- (0,0) -- (-5.5,4.125);
        
        \node at (0,5){$\bullet$};
        \node at (1,5){$\bullet$};
        \node at (-1,5){$\bullet$};
        \node at (2,4.5){$\bullet$};
        \node at (2,5.5){$\bullet$};
        \node at (-2,4.5){$\bullet$};
        \node at (-2,5.5){$\bullet$};
        
        \node at (0,4.8){$g$};
        \node at (-1,4.8){$a$};
        \node at (-2.2,4.5){$b$};
        \node at (-2.2,5.5){$c$};
        \node at (1,4.8){$d$};
        \node at (2.2,4.5){$e$};
        \node at (2.2,5.5){$f$};
        
        \node at (0,.3){$g$};
        \node at (-1.3,.75){$a$};
        \node at (-2.3,1.5){$b$};
        \node at (-3.3,2.25){$c$};
        \node at (-4.3,3){$a$};
        \node at (-5.3,3.75){$b$};
        
        \node at (1.3,.75){$d$};
        \node at (2.3,1.5){$e$};
        \node at (3.3,2.25){$f$};
        \node at (4.3,3){$d$};
        \node at (5.3,3.75){$e$};
        
        \draw (-1,5) -- (-2,5.5) -- (-2,4.5) -- (-1,5) -- (1,5) -- (2,4.5) -- (2,5.5) -- (1,5);
        
        \draw[dashed] (-3,0) -- (3,4.5);
        
        \node at (-.6,2.25){$F \uparrow$};
        \node at (.5,2.25){$I \downarrow$};
    \end{tikzpicture}
    \vs 
    \caption{Simple graph $\Gamma$ and poset $P$ colored by $\Gamma$.  Split $(F,I)$ is used to illustrate computation of $\{\upsilon_i\}_{i \in \Gamma}$.}
    \label{420}
\end{figure}

\begin{Def} \label{Mu}
Let $(F,I)$ be a split and let $b \in \Gamma$.  If $P_b \cap I \ne \emptyset$, then define $\mu_b(F,I) := 1 - \upsilon_b(F,I)$.  If $P_b \cap I = \emptyset$, then define $\mu_b(F,I) := -1 + \psi_b(F,I)$.  Finally, define the $\Gamma$-set of operators $\{M_a\}_{a \in \Gamma}$ to be the diagonal operators with $\mf{h}'$-weight $\{\mu_a\}_{a \in \Gamma}$.  These are the \emph{$\mu$-diagonal operators}. 
\end{Def}

Foreshadowing Proposition \ref{458}(c), we see that $M_b.\la F,I \ra = +\la F,I \ra$ if $I$ has a maximal element of color $b$ since then $\upsilon_b(F,I) = 0$.  
\ni The definition of $\{\mu_a\}_{a \in \Gamma}$ is not symmetric with respect to $F$ and $I$; an alternate construction can be made for the $\mu$-diagonal operators that emphasizes filters instead of ideals.  Let $b \in \Gamma$ and $(F,I) \in \mc{FI}(P)$.  If $P_b \cap F = \emptyset$, then set $\mu'_b(F,I) := 1 - \upsilon_b(F,I)$.  If $P_b \cap F \ne \emptyset$, then set $\mu'_b(F,I) := -1 + \psi_b(F,I)$.  When $P$ satisfies EC, AC, and I2A (defined next), it can be shown that $\mu_b' = \mu_b$; see \cite[Prop. 4.1.1]{Str}.

\begin{Prop} \label{423}
Suppose $P$ satisfies EC and the following additional properties:
\bi[nosep]
\i (AC): Elements with adjacent colors are comparable, and
\i (I2A): For every $a \in \Gamma$: The open interval between any two consecutive elements of color $a$ contains exactly two elements whose colors are adjacent to $a$.
\ei
Then we have the following:
\begin{enumerate}[(a),nosep]
\item The property I2A implies both ND and I3ND.
\item The weight function $\{\mu_a\}_{a \in \Gamma}$ is an edge weight function.  
\item If $P$ also satisfies NA, then the $\mu$-diagonal operators $\{M_a\}_{a \in \Gamma}$ can be used to extend the representations of Theorem \ref{351} from $\mf{n}_+$ and $\mf{n}_-$ to $X$- and $Y$-square nilpotent representations of $\mf{b}'_+$ and $\mf{b}'_-$.
\end{enumerate}
\end{Prop} 


\bp 
Suppose $P$ fails either ND or I3ND.  Then $P$ contains an interval that is a chain consisting of two or three elements in which the minimal and maximal elements are consecutive elements of the same color.  The open interval between these two elements violates I2A. Thus (a) holds. 

To prove (b), 
fix $b \in \Gamma$.  Let $(F,I) \in \mc{FI}(P)$ and suppose $x$ is minimal in $F$.  Set $a := \kappa(x)$.  We must show that $\mu_b(F-x,I+x) - \mu_b(F,I) = \theta_{ab}$.

First suppose $a = b$.  We start with the case $P_a \cap I = \emptyset$.  Since $x$ is minimal in $F$, we have $\psi_a(F,I) = |\Psi_a(F,I)| = 0$ and $\mu_a(F,I) = -1$.  Here $P_a \cap (I+x) \ne \emptyset$.  Since $x$ is maximal in $I+x$, we have $\upsilon_a(F-x,I+x) = |\Upsilon_a(F-x,I+x)| = 0$ and $\mu_a(F-x,I+x) = 1$.  We get $\mu_a(F-x,I+x) - \mu_a(F,I) = 2 = \theta_{aa}$.  Otherwise we have the case $P_a \cap I \ne \emptyset$.  
Let $z \in P_a \cap I$.  Note that $z < x$ by EC and that $P_a \cap [z,x]$ is finite by local finiteness for $[z,x]$.  So $P_a \cap I$ has a maximal element $y$.
Here $y < x$ are consecutive occurrences of the color $a$.  By I2A there are exactly two elements $u,v \in (y,x)$ with colors adjacent to $a$.  By AC all elements greater than $y$ in $I$ with colors adjacent to $a$ are in $(y,x)$.  Hence $u$ and $v$ are the only such elements.  This shows both $u$ and $v$ are in $\Upsilon_a(F,I)$, and no other elements can be in $\Upsilon_a(F,I)$.  Thus $\upsilon_a(F,I) = |\Upsilon_b(F,I)| = 2$ and $\mu_a(F,I) = -1$.  We still have $\mu_a(F-x,I+x) = 1$, and so again $\mu_a(F-x,I+x) - \mu_a(F,I) = 2 = \theta_{aa}$.

Now suppose $a \adj b$.  We again start with the case $P_b \cap I = \emptyset$.  Here $P_b \cap (I+x) = \emptyset$ as well.  We know $P_b \cap F \ne \emptyset$ since $\kappa$ is surjective.  Let $z \in P_b \cap F$.  Note that $x < z$ by AC and that $[x,z]$ and $P_b \cap [x,z]$ are finite.  So $P_b \cap F$ has a minimal element $y$.  Then $x$ satisfies all three criteria to be in $\Psi_b(F,I)$, but $x \notin \Psi_b(F-x,I+x)$ since $x \notin F-x$.  Observe that $y$ is also minimal in $P_b \cap (F-x)$.  So we have $\Psi_b(F-x,I+x) = \Psi_b(F,I) - \{x\}$.  Thus $|\Psi_b(F-x,I+x)| = |\Psi_b(F,I)| - 1$.  Note that there is some color $c \adj b$ such that there are infinitely many elements less than $y$ in $P_c \cap F$ if and only if the same statement is true for $F-x$.  Thus whether or not such a color $c$ exists we have $\psi_b(F-x,I+x) = \psi_b(F,I) - 1$.  Here $\mu_b(F-x,I+x) = -1 + \psi_b(F-x,I+x)$ and $\mu_b(F,I) = -1 + \psi_b(F,I)$.  Thus $\mu_b(F-x,I+x) - \mu_b(F,I) = -1 = \theta_{ab}$.

Otherwise for $a \adj b$ we have the case $P_b \cap I \ne \emptyset$.  Here $P_b \cap (I + x) \ne \emptyset$ as well.  Let $z \in P_b \cap I$.  Note $z < x$ by AC and that $[z,x]$ and $P_b \cap [z,x]$ are finite.  So $P_b \cap I$ has a maximal element $y$.  Note that $x$ satisfies all three criteria to be in $\Upsilon_b(F-x,I+x)$, but $x \notin \Upsilon_b(F,I)$ since $x \notin I$.  Observe that $y$ is also maximal in $P_b \cap (I + x)$.  So we have $\Upsilon_b(F-x,I+x) = \Upsilon_b(F,I) \cup \{x\}$.  Thus $|\Upsilon_b(F-x,I+x)| = |\Upsilon_b(F,I)| + 1$.  Note that there is some color $c \adj b$ such that there are infinitely many elements greater than $y$ in $P_c \cap I$ if and only if the same statement is true for $I + x$.  Thus whether or not such a color $c$ exists we have $\upsilon_b(F-x,I+x) = \upsilon_b(F,I) + 1$.  Here $\mu_b(F-x,I+x) = 1 - \upsilon_b(F-x,I+x)$ and $\mu_b(F,I) = 1 - \upsilon_b(F,I)$.  Thus $\mu_b(F-x,I+x) - \mu_b(F,I) = -1 = \theta_{ab}$.

Finally suppose $a \dist b$.  We again start with the case $P_b \cap I = \emptyset$.  Here $P_b \cap (I + x) = \emptyset$ as well.  An element $y$ is minimal in $P_b \cap F$ if and only if it is minimal in $P_b \cap (F - x)$.  If no such minimal element exists, then $\psi_b(F,I) = 1 = \psi_b(F-x,I+x)$.  Otherwise there is an element $y$ minimal in both $P_b \cap F$ and $P_b \cap (F - x)$.  Note that there is a color $c \adj b$ such that there are infinitely many elements less than $y$ in $P_c \cap F$ if and only if the same statement is true for $P_c \cap (F-x)$.  Thus whether or not such a color $c$ exists, since $a \not \simeq b$ we have $\psi_b(F-x,I+x) = \psi_b(F,I)$.  Hence $\mu_b(F-x,I+x) - \mu_b(F,I) = 0 = \theta_{ab}$.  Otherwise for $a \dist b$ we have the case $P_b \cap I \ne \emptyset$.  In this case, we replace $\psi_b$ with $\upsilon_b$ and dualize to again get $\mu_b(F-x,I+x) - \mu_b(F,I) = 0 = \theta_{ab}$.

Thus $\{\mu_a\}_{a \in \Gamma}$ is an edge weight function, and so (b) holds.  Since I2A implies I3ND, we can apply Proposition \ref{380}(b) to get (c).
\ep

\section{Existence and uniqueness for $\boldsymbol{\mf{sl}_2}$ weights along color strings}\label{453} 

The actions of the $\{h_a\}_{a \in \Gamma}$ in a minuscule representation of a semisimple Lie algebra have certain values along their ``$\mf{sl}_2$ strings.''  To obtain upper $P$-minuscule representations of $\mf{b}'_+$ and $P$-minuscule representations of $\mf{g}'$, we need edge weight functions that have these values along the ``color strings'' of $\mc{FI}(P)$.  The next result is the first step toward obtaining these values.  This existence result motivates the properties AC and I2A from a Lie representation viewpoint.

\begin{Prop} \label{458} (a) The following are equivalent:
\begin{enumerate}[(a),nosep]
\item[] 
\be[(i),nosep]
\item There exists an edge weight function $\{\eta_a\}_{a \in \Gamma}$ such that for every $b \in \Gamma$ and every split $(F,I)$, we have $\eta_b(F,I) = -1$ if $b$ is the color of a minimal element of $F$.
\item The properties EC, AC, and I2A are satisfied by $P$.
\item There exists an edge weight function $\{\eta_a\}_{a \in \Gamma}$ such that for every $b \in \Gamma$ and every split $(F,I)$, we have $\eta_b(F,I) = +1$ if $b$ is the color of a maximal element of $I$.
\ee
\setcounter{enumi}{1}
\item Any choice of edge weight function satisfying Condition (i) will work for (iii) (and vice versa).
\item If Condition (ii) is satisfied, then the edge weight function $\{\mu_a\}_{a \in \Gamma}$ of Definition \ref{Mu} satisfies Conditions (i) and (iii). 
\end{enumerate}
\end{Prop}

\begin{Rmk} 
For each $a \in \Gamma$, Condition (i) (or (iii)) above requires certain values for the function $\eta_a$ along the edges of color $a$ in $\mc{FI}(P)$.  But when $\mc{FI}(P)$ has a component that does not contain an edge of some color $b \in \Gamma$, those parts do not pertain to $\eta_b$ on that component.  For example, the component containing $(F,I)$ in Figure \ref{420} does not contain an edge colored $g$.  So there $\eta_g(F,I)$ can be any complex number.
\end{Rmk}

\bp[Proof of Proposition \ref{458}.]
For Part (a), we first show (iii) implies (i).  Create an edge weight function $\{\eta_a\}_{a \in \Gamma}$ that satisfies (iii).  Let $b \in \Gamma$ and $(F,I) \in \mc{FI}(P)$.  Suppose $b$ is the color of a minimal element $y$ of $F$.  Then $\eta_b(F-y,I+y) = 1$ by (iii) since $y$ is maximal in $I+y$.  
Thus Equation (\ref{373}) gives $\eta_b(F-y,I+y) - \eta_b(F,I) = 2$.  Hence we get $\eta_b(F,I) = \eta_b(F-y,I+y) - 2 = -1$, yielding (i).
A dual argument shows that (i) implies (iii).  
This also shows that one choice will work for both Parts (i) and (iii), so (b) holds.

We next show (iii) implies (ii).  Continue to consider the edge weight function $\{\eta_a\}_{a \in \Gamma}$ above.  Let $x$ and $y$ be incomparable elements in $P$.  Define $a := \kappa(x)$ and $b := \kappa(y)$.  Let $F$ be the filter generated by $x$ and $y$ and set $I := P - F$.  Note that $b$ is the color of a maximal element of both $I + y$ and $I + y + x$.  Thus by (iii) we have $\eta_b(F-y,I+y) = 1 = \eta_b(F - y - x, I + y + x)$, so $\eta_b(F-y-x,I+y+x) - \eta_b(F-y,I+y) = 0$.  Since $x$ is minimal in $F-y$, we can apply Equation (\ref{373}) to also obtain $\eta_b(F-y-x,I+y+x) - \eta_b(F-y,I+y) = \theta_{ab}$.  Thus $\theta_{ab} = 0$, and so $a \dist b$.  Thus we get both EC and AC.  Now let $b \in \Gamma$ and suppose that $x < y$ are consecutive occurrences of the color $b$.  Define $I'$ to be the principal ideal generated by $y$.  Define the ideal $I$ to be $I' - (x,y]$, where $(x,y] := \{z \in P \ | \ x < z \le y\}$; also note that $x$ is maximal in $I$ and that $I' - I = (x,y]$ and $I - I' = \emptyset$.  Define $F' := P - I'$ and $F := P - I$.  Since $(x,y]$ is finite, the splits $(F',I')$ and $(F,I)$ are in the same component of $\mc{FI}(P)$.  Since $y$ is maximal in $I'$ and $x$ is maximal in $I$, we have $\eta_b(F',I') = 1 = \eta_b(F,I)$.  
From Equation (\ref{371}), which we may use by Proposition \ref{377}, we get
\begin{align*}
0 = \eta_b(F',I') - \eta_b(F,I) &= 2\Delta_b[(F',I'),(F,I)] - \sum_{c \adj b} \Delta_c[(F',I'),(F,I)] \\
&= 2|P_b \cap (x,y]| - \sum_{c \adj b} |P_c \cap (x,y]|.
\end{align*}
\ni Since $\kappa(y) = b$, this equation can be rewritten $2 = \sum_{c \adj b} |P_c \cap (x,y)|$.  Thus I2A holds.

Now assume (ii) holds and consider the weight function $\{\mu_a\}_{a \in \Gamma}$ of Definition \ref{Mu}.  By Proposition \ref{423}(b) we know that $\{\mu_a\}_{a \in \Gamma}$ is an edge weight function.  Fix a color $b \in \Gamma$ and a split $(F,I) \in \mc{FI}(P)$, and suppose $b$ is the color of a maximal element of $I$.  Then we have $\upsilon_b(F,I) = 0$ and $\mu_b(F,I) = +1$, so (iii) holds.  Then by (b) we see that (c) holds.
\ep

We do get uniqueness for $\{\mu_a\}_{a \in \Gamma}$ on a component when this component has edges of all colors:

\begin{Cor} \label{463}
Suppose $P$ satisfies EC, AC, and I2A.  Fix a component $\mc{C}$ and suppose there is an edge in $\mc{C}$ of every color.  Then $\{\mu_a\}_{a \in \Gamma}$ is the unique restriction to $\mc{C}$ of an edge weight function that satisfies Conditions (i) or (iii) of Proposition \ref{458}(a).  So if $P$ is finite, then $\{\mu_a\}_{a \in \Gamma}$ is the unique edge weight function on $\mc{FI}(P)$ that satisfies these conditions.
\end{Cor}

\bp
Let $\{\eta_a\}_{a \in \Gamma}$ be an edge weight function that satisfies Condition (i) of Proposition \ref{458}(a), and let $b \in \Gamma$.  Since there is an edge in $\mc{C}$ of color $b$, let $(F_0,I_0) \in \mc{C}$ be such that $F_0$ has a minimal element $y$ of color $b$.  We know $\eta_b(F_0,I_0) = -1$.  Then Proposition \ref{458}(c) says $\mu_b(F_0,I_0) = -1$ as well.  Now $\{\eta_a\}_{a \in \Gamma}$ and $\{\mu_a\}_{a \in \Gamma}$ are component weight functions by Proposition \ref{377}, so Lemma \ref{382}(b) says that $\eta_b$ and $\mu_b$ agree on all of $\mc{C}$.  A similar argument holds when $\{\eta_a\}_{a \in \Gamma}$ satisfies Condition (iii) of Proposition \ref{458}(a).  When $P$ is finite, the lattice $\mc{FI}(P)$ has one component by Lemma \ref{366}.
\ep 

\section{Frontier census properties}\label{522}

Here we introduce our last coloring properties.  For each $k \ge 1$ we define two \emph{frontier census properties}:
\bi[nosep]
\i (Mx$k$GA): For every color $a \in \Gamma$: If $x$ is maximal in $P_a$, then there are at most $k$ elements greater than $x$ that have their colors adjacent to $a$,
\i (Mn$k$LA): For every color $a \in \Gamma$: If $x$ is minimal in $P_a$, then there are at most $k$ elements less than $x$ that have their colors adjacent to $a$.
\ei
\ni The properties Mx1GA and Mn1LA are the most important of these properties; in \cite[Ch. 8]{Str} the property Mn2LA was also used.  In \cite[\S 7.4]{Str} we indicated how Mx1GA and Mn1LA revamp, generalize, and unify axioms considered by Stembridge and Green.  The property Mx1GA was retrospectively found to be implicitly present in Proposition 2.5 of \cite{Ste}.  That early statement in \cite{Ste} was formulated in terms of decompositions of Weyl group elements $w$, before the heap finite colored posets were introduced.  

The frontier census properties limit the eigenvalues of the coroot actions; see \cite[\S 4.3]{Str} for interactions between these properties and the edge weight function $\{\mu_a\}_{a \in \Gamma}$ of Definition \ref{Mu}.

\section{Upper $\boldsymbol{P}$-minuscule representations of $\boldsymbol{\mf{b}'_+}$}\label{609} 

The first main result of this paper gives necessary and sufficient conditions on coloring properties for $P$ so that $\mc{FI}(P)$ carries an upper $P$-minuscule representation of $\mf{b}'_+$; see Definition \ref{MinReps}(a). 

\begin{Thm}\label{614}
Let $P$ be a poset whose elements are colored by the nodes of a finite simple graph $\Gamma$.  Let $\mc{FI}(P)$ be the lattice of filter-ideal splits of $P$.  
\begin{enumerate}[(a),nosep]
\item The following are equivalent:
\be[(i),nosep]
\item The lattice $\mc{FI}(P)$ carries an upper (respectively lower) $P$-minuscule representation of $\mf{b}'_+$ (respectively $\mf{b}'_-$).
\item The poset $P$ satisfies EC, NA, AC, I2A, and Mx1GA (respectively Mn1LA).
\ee
\item If Condition (ii) is satisfied, then the $\mu$-diagonal operators $\{M_a\}_{a \in \Gamma}$ of Definition \ref{Mu} can be used to give the actions of the $\{h_a\}_{a \in \Gamma}$ for the upper (respectively lower) $P$-minuscule representation of $\mathfrak{b}'_+$ (respectively $\mathfrak{b}'_-$) of (i).
\end{enumerate}
\end{Thm}

\begin{Rmk} 
It can be confirmed that the poset displayed in Figure \ref{289} satisfies the Mx1GA version of Condition (ii); the poset displayed in Figure \ref{134} satisfies both versions.  
If $P$ is connected and satisfies Condition (ii), then $\mathcal{FI}(P)$ has only one nontrivial component.  Then the $\mu$-diagonal operators are the unique operators satisfying Condition (i) on this component.
This is Corollary 8.3.9 of \cite{Str}, which followed from the classification of posets satisfying Condition (ii) presented there in Section 8.3.
\end{Rmk}

\bp[Proof of Theorem \ref{614}.]
To show (a), assume that (i) holds for the terms upper and $\mathfrak{b}'_+$ and Mx1GA.  Since this representation is $X$-square nilpotent, by Theorem \ref{388} we know $P$ satisfies EC and NA.  Let $\{H_a\}_{a \in \Gamma}$ be the diagonal operators for this representation with $\mf{h}'$-weight $\{\eta_a\}_{a \in \Gamma}$ and eigenvalue set $\mc{E}_{\eta} := \{\eta_a(F,I) \ | \ a \in \Gamma, (F,I) \in \mc{FI}(P)\}$.  By Proposition \ref{380}(a) we know that $\{\eta_a\}_{a \in \Gamma}$ is an edge weight function.  Since the representation is upper $P$-minuscule, for all $a \in \Gamma$ and $(F,I) \in \mc{FI}(P)$ we have
\begin{align}
\label{617} &\eta_a(F,I) \in \mc{E}_\eta \subseteq \{-1,0,1,2,\dots\}, \\
\label{618} &\eta_a(F,I) = -1 \ \t{if and only if $a$ is the color of a minimal element of $F$}.
\end{align}
\ni By (\ref{618}) and Proposition \ref{458}(a) we see $P$ satisfies AC and I2A.  For the sake of contradiction, suppose Mx1GA fails.  Then there is a color $b \in \Gamma$ and an element $y$ maximal in $P_b$ such that there are two or more elements greater than $y$ with colors adjacent to $b$.  Let $u$ and $v$ be two such elements.  Define $I'$ to be the ideal generated by $u$ and $v$.  Define $I := I' - ((y,u] \cup (y,v])$.  Note $I$ is an ideal of $P$ and that $y$ is maximal in $I$.  Define $F' := P - I'$ and $F := P - I$, and further note that $I' - I = (y,u] \cup (y,v]$ and $I - I' = \emptyset$.  By local finiteness, the splits $(F',I')$ and $(F,I)$ are in the same component of $\mc{FI}(P)$.  By Proposition \ref{377} we know that $\{\eta_a\}_{a \in \Gamma}$ is a component weight function, so using Equation (\ref{371}) we get
\begin{align*}
\eta_b(F',I') - \eta_b(F,I) &= 2\Delta_b[(F',I'),(F,I)] - \sum_{c \adj b} \Delta_c[(F',I'),(F,I)] \\
&= 2|P_b \cap ((y,u] \cup (y,v])| - \sum_{c \adj b} |P_c \cap ((y,u] \cup (y,v])|. 
\end{align*}
\ni Since $y$ is maximal in $P_b$, we have $|P_b \cap ((y,u] \cup (y,v])| = 0$.  Since $u$ and $v$ have colors adjacent to $b$, we have $\sum_{c \adj b} |P_c \cap ((y,u] \cup (y,v])| \ge 2$.  Thus we get the inequality $\eta_b(F',I') - \eta_b(F,I) \le -2$.  Since $y$ is maximal in $I$ and $\kappa(y) = b$, by Proposition \ref{458}(b) we get $\eta_b(F,I) = 1$.  Thus the inequality becomes $\eta_b(F',I') \le -1$.  By (\ref{617}) we have $\eta_b(F',I') \ge -1$, and so $\eta_b(F',I') = -1$.  Then (\ref{618}) shows $F'$ has a minimal element $z$ of color $b$.  Since $y \in I'$, by EC we have $y < z$.  But $y$ is maximal in $P_b$.  This contradiction shows Mx1GA holds, and so (ii) holds.

Now assume that (ii) holds.  Since $P$ satisfies EC, AC, and I2A, we know by Proposition \ref{423}(b) that $\{\mu_a\}_{a \in \Gamma}$ is an edge weight function.  We use the operators $\{M_a\}_{a \in \Gamma}$ specified by $\{\mu_a\}_{a \in \Gamma}$ to get the desired actions of the $\{h_a\}_{a \in \Gamma}$.  Since $P$ also satisfies NA, Proposition \ref{423}(c) says $\mc{FI}(P)$ carries an $X$-square nilpotent representation of $\mf{b}'_+$ when using the operators $\{M_a\}_{a \in \Gamma}$.  The eigenvalue set $\mc{E}_\mu$ of these operators is contained in $\mb{Z}$ by construction.  Let $b \in \Gamma$ and $(F,I) \in \mc{FI}(P)$, and for the sake of contradiction, suppose $\mu_b(F,I) < -1$.  Then necessarily we have $P_b \cap I \ne \emptyset$, and so $\mu_b(F,I) = 1 - \upsilon_b(F,I)$ and $\upsilon_b(F,I) \ge 3$.  Thus $P_b \cap I$ must have a maximal element $y$ of color $b$, and there must be three or more elements greater than $y$ in $I$ with colors adjacent to $b$.  Then by AC and I2A we see $P_b \cap F = \emptyset$, so $y$ is maximal in $P_b$.  But this would violate Mx1GA.  Thus we see that $\mu_b(F,I) \ge -1$, and so $\mc{E}_\mu \subseteq \{-1,0,1,2,\dots\}$.  Then Proposition \ref{458}(c) says $\mu_b(F,I) = -1$ if $b$ is the color of a minimal element of $F$.  

We now need to show that if $\mu_b(F,I) = -1$, then $b$ is the color of a minimal element of $F$.  Hence suppose that $\mu_b(F,I) = -1$.  We start with the case $P_b \cap I \ne \emptyset$.  Then $-1 = \mu_b(F,I) = 1 - \upsilon_b(F,I)$, so $\upsilon_b(F,I) = 2$.  This shows $P_b \cap I$ has a maximal element $y$ since $\upsilon_b(F,I) = 1$ otherwise.  Let $c \in \Gamma$ be such that $c \adj b$.  If there are infinitely many elements greater than $y$ in $P_c \cap I$, then $P_b \cap F$ is empty by AC and local finiteness.  But then $y$ would be maximal in $P_b$, which would violate Mx1GA.  Thus no such $c$ exists, and so $|\Upsilon_b(F,I)| = \upsilon_b(F,I) = 2$.  So there are two distinct elements $u$ and $v$ in $I$ greater than $y$ with colors adjacent to $b$.  By Mx1GA we know that $y$ cannot be maximal in $P_b$.  Let $z \in P_b$ be such that $y < z$ are consecutive occurrences of the color $b$.  Since $y$ is maximal in $P_b \cap I$, we have $z \in F$.  By AC we know $u,v \in (y,z)$.  Suppose $w \in P$ satisfies $w \to z$.  Then NA implies that $\kappa(w) \adj b$.  Thus by AC we have $w \in (y,z)$.  By I2A we know that $w = u$ or $w = v$, so $w \in I$.  Thus $z$ is minimal in $F$ and has color $b$.  

Otherwise we have the case $P_b \cap I = \emptyset$.  Then $-1 = \mu_b(F,I) = -1 + \psi_b(F,I)$, so $\psi_b(F,I) = 0$.  This shows $P_b \cap F$ has a minimal element $y$ since $\psi_b(F,I) = 1$ otherwise.  
Suppose $w \in P$ satisfies $w \to y$.  By NA we know that $\kappa(w) \sim b$.  We also know there are finitely many elements less than $y$ in $P_{\kappa(w)} \cap F$ since otherwise $\psi_b(F,I) \ge 1$.  If $w \in F$, this shows that $w \in \Psi_b(F,I)$.  But since $0 = \psi_b(F,I) \ge |\Psi_b(F,I)|$, we have $w \notin F$.
Hence $y$ is minimal in $F$ and has color $b$.  Thus $\mu_b(F,I) = -1$ if and only if $b$ is the color of a minimal element of $F$.  Hence the representation of $\mf{b}'_+$ carried by $\mc{FI}(P)$ with the actions of $\{h_a\}_{a \in \Gamma}$ given by $\{M_a\}_{a \in \Gamma}$ is upper $P$-minuscule, so (i) holds.

The equivalence of (i) and (ii) for lower and $\mathfrak{b}'_-$ and Mn1LA follows from a dualized argument, again using the $\mu$-diagonal operators $\{M_a\}_{a \in \Gamma}$ for the proof that (ii) implies (i).  Thus (a) holds.  
In both cases of (ii) implies (i), the $\mu$-diagonal operators were used to build the required representation.  Thus (b) holds.
\ep

The following lemma is used for Corollary \ref{627} and Theorem \ref{653}.

\begin{Lem}\label{619}
Let $\{H_a\}_{a \in \Gamma}$ be diagonal operators on $\langle \mathcal{FI}(P) \rangle$.  Suppose the operators $\{X_a,H_a\}_{a \in \Gamma}$ (respectively $\{Y_a,H_a\}_{a \in \Gamma}$) satisfy HX (HY).  If the eigenvalue set of $\{H_a\}_{a \in \Gamma}$ is contained in $\{-1,0,1\}$, then the actions of the color raising (lowering) operators $\{X_a\}_{a \in \Gamma}$ (respectively $\{Y_a\}_{a \in \Gamma}$) are $X$-square ($Y$-square) nilpotent. 
\end{Lem}

\begin{proof}
Suppose for a contradiction that there is some $a \in \Gamma$ and $(F,I) \in \mathcal{FI}(P)$ such that $X_a^2.\langle F,I \rangle \ne 0$.  Using the relation $H_aX_a = X_aH_a + [H_a,X_a] = X_aH_a + 2X_a$ obtained from HX twice, we get
\begin{align*}  
H_a X_a^2.\langle F,I \rangle = (X_aH_aX_a + 2X_a^2).\langle F,I \rangle = (X_a^2H_a + 2X_a^2 + 2X_a^2).\langle F,I \rangle. 
\end{align*}
\noindent Since the eigenvalue of $H_a.\langle F,I \rangle$ is either $-1$, $0$, or $1$, the right hand side is $\xi X_a^2. \langle F,I \rangle$ for some $\xi \in \{3,4,5\}$.  But since all basis vectors in the expansion of $X_a^2.\langle F,I \rangle$ are $\mathfrak{h}'$-weight vectors with eigenvalues in $\{-1,0,1\}$, this is impossible.  Thus this representation is $X$-square nilpotent.  The proof for $\{Y_a,H_a\}_{a \in \Gamma}$ follows from a dualized argument.
\end{proof}

\ni The following corollary will be expanded in Theorem \ref{653}, which is our foremost main result.

\begin{Cor}\label{627}
The following are equivalent:
\begin{enumerate}[(i),nosep]
\item \label{(ii)} The lattice $\mathcal{FI}(P)$ carries upper and lower $P$-minuscule representations of $\mathfrak{b}'_+$ and $\mathfrak{b}'_-$, respectively, using the same diagonal operators $\{H_a\}_{a \in \Gamma}$.
\item \label{(iv)} The lattice $\mc{FI}(P)$ carries representations of $\mf{b}'_+$ and $\mf{b}'_-$ with diagonal operators $\{H_a\}_{a \in \Gamma}$ that satisfy for every $b \in \Gamma$ and every split $(F,I)$:
\bi[nosep]
\i $H_b.\la F,I \ra = -\la F,I \ra$ if $b$ is the color of a minimal element of $F$,
\i $H_b.\la F,I \ra = +\la F,I \ra$ if $b$ is the color of a maximal element of $I$, and
\i $H_b.\la F,I \ra = 0$ otherwise.
\ei
\item \label{(v)} The poset $P$ satisfies the properties EC, NA, AC, I2A, Mx1GA, and Mn1LA.
\end{enumerate}
When any of these conditions are satisfied, the $\mu$-diagonal operators $\{M_a\}_{a \in \Gamma}$ are the unique operators satisfying \ref*{(ii)} or \ref*{(iv)}.
\end{Cor}

\bp 
Suppose \ref*{(iv)} holds with diagonal operators $\{H_a\}_{a \in \Gamma}$.  The eigenvalue set of each representation is contained in $\{-1,0,1\}$.  By Lemma \ref{619} these representations of $\mf{b}'_+$ and $\mf{b}'_-$ are respectively $X$- and $Y$-square nilpotent.  The eigenvalue condition in \ref*{(iv)} shows for all $b \in \Gamma$ and $(F,I) \in \mc{FI}(P)$ that $H_b.\la F,I \ra = -\la F,I \ra$ (respectively $H_b.\la F,I \ra = +\la F,I \ra$) if and only if $b$ is the color of a minimal (maximal) element of $F$ (respectively $I$).  Thus the representation of $\mf{b}'_+$ (respectively $\mf{b}'_-$) is upper (lower) $P$-minuscule, so \ref*{(ii)} holds. 

Now suppose \ref*{(ii)} holds with diagonal operators $\{H_a\}_{a \in \Gamma}$.  Theorem \ref{614}(a) shows \ref*{(v)} holds.  Since the representations of $\mf{b}'_+$ and $\mf{b}'_-$ are respectively upper and lower $P$-minuscule, the eigenvalue set of $\{H_a\}_{a \in \Gamma}$ is contained in $\{-1,0,1,2,\dots\} \cap \{\dots,-2,-1,0,1\} = \{-1,0,1\}$.  Also, for $b \in \Gamma$ and $(F,I) \in \mc{FI}(P)$ we have $H_b.\la F,I \ra = -\la F,I \ra$ (respectively $H_b.\la F,I \ra = +\la F,I \ra$) if and only if $F$ (respectively $I$) has a minimal (maximal) element of color $b$.  Hence $H_b.\la F,I \ra = 0$ otherwise, so \ref*{(iv)} holds. 

Now suppose \ref*{(v)} holds.  Then Theorem \ref{614}(b) shows \ref*{(ii)} holds using the $\mu$-diagonal operators $\{M_a\}_{a \in \Gamma}$. 

The conditions in \ref*{(iv)} completely specify the operators, so they are unique.  The $\mu$-diagonal operators satisfy \ref*{(ii)} when these conditions hold, and this proof showed that any diagonal operators satisfying \ref*{(ii)} also satisfy \ref*{(iv)}.
\ep 

\section{$\boldsymbol{P}$-minuscule representations of $\boldsymbol{\mf{g}'}$}\label{631}

We obtain our foremost main result, Theorem \ref{653}.  It provides necessary as well as sufficient conditions on coloring properties for $P$ needed for $\mc{FI}(P)$ to carry a $P$-minuscule representation of $\mf{g}'$.  Both its statement and its proof simultaneously handle finite and infinite posets.

\begin{Lem}\label{645}
Suppose $P$ satisfies EC.  The relation $[X_b,Y_a] = 0$ holds when $a,b \in \Gamma$ are distinct. 
\end{Lem}

\begin{proof}
Fix distinct $a,b \in \Gamma$ and let $(F,I) \in \mathcal{FI}(P)$.  Suppose that $X_bY_a. \langle F,I \rangle \ne 0$.  Then with EC in mind, there exist unique elements $x,y \in P$ such that $\kappa(y) = a$ and $\kappa(x) = b$ satisfying $X_bY_a.\langle F,I \rangle = \langle (F+y) - x, (I - y) + x\rangle$.  Note that both $x$ and $y$ are minimal in the filter $F+y$.  Since $x \ne y$, this shows that $x$ and $y$ are incomparable.  Since $x$ is minimal in $F + y$, it is also minimal in $F$.  Since $x$ and $y$ are incomparable and $y$ is maximal in $I$, it is also maximal in $I + x$.  Hence $Y_aX_b.\langle F,I \rangle = \langle (F-x)+y,(I+x)-y \rangle = X_bY_a.\langle F,I \rangle$.   Hence $[X_b,Y_a].\la F,I \ra = 0$.  The case where $Y_aX_b.\la F,I \ra \ne 0$ follows from a dualized argument.
\end{proof}


Here we characterize the $P$-minuscule representations of $\mf{g}'$ in several ways; see Section \ref{251} for the definitions of $X$- and $Y$-square nilpotent actions and Definition \ref{MinReps}(b) for the definition of $P$-minuscule representations of $\mf{g}'$.  

\begin{Thm}\label{653}
Let $P$ be a poset whose elements are colored by the nodes of a finite simple graph $\Gamma$.  Let $\mc{FI}(P)$ be the lattice of filter-ideal splits of $P$.  Then (i) and (ii) from Corollary \ref{627} are each equivalent to:
\be[(i),nosep]
\setcounter{enumi}{2}
\item The poset $P$ satisfies the properties EC, NA, AC, I2A, Mx1GA, and Mn1LA.
\item \label{(iii)} The lattice $\mc{FI}(P)$ carries $X$- and $Y$-square nilpotent representations of $\mf{b}'_+$ and $\mf{b}'_-$ respectively, using diagonal operators $\{H_a\}_{a \in \Gamma}$ that satisfy $[X_a,Y_a] = H_a$ for every color $a \in \Gamma$.
\item \label{(i)} The lattice $\mc{FI}(P)$ carries a $P$-minuscule representation of $\mf{g}'$.
\ee
\ni When any of these conditions are satisfied, the $\mu$-diagonal operators $\{M_a\}_{a \in \Gamma}$ of Definition \ref{Mu} are the unique diagonal operators satisfying 
\ref*{(ii)}, \ref*{(iv)}, \ref*{(iii)}, or \ref*{(i)}. 
\end{Thm}

\ni It can be confirmed that the poset displayed in Figure \ref{134} satisfies Condition \ref*{(v)}.

\bp
Suppose that \ref*{(i)} holds.  Then diagonal operators $\{H_a\}_{a \in \Gamma}$ exist so that XX, YY, HX, HY, and XY hold.  The eigenvalue set of $\{H_a\}_{a \in \Gamma}$ is contained in $\{-1,0,1\}$.  By restricting to $\{X_a,H_a\}_{a \in \Gamma}$, we know $\mc{FI}(P)$ carries a representation of $\mf{b}'_+$.  By Lemma \ref{619} we know this representation is $X$-square nilpotent.  By Proposition \ref{323}, we know the representation of $\mf{b}'_-$ obtained by restricting to $\{Y_a,H_a\}_{a \in \Gamma}$ is $Y$-square nilpotent.  Since XY(i) is the relation $[X_a,Y_a] = H_a$ for all $a \in \Gamma$, we see that \ref*{(iii)} holds.

Now suppose \ref*{(iii)} holds for some such diagonal operators $\{H_a\}_{a \in \Gamma}$.  The relations XX, YY, HX, and HY are satisfied.  Since the representations of $\mf{b}'_+$ and $\mf{b}'_-$ are respectively $X$- and $Y$-square nilpotent, Proposition \ref{323} shows $P$ satisfies EC and ND.  Lemma \ref{645} and the assumed relation $[X_a,Y_a] = H_a$ for all $a \in \Gamma$ show XY holds.  Thus $\mc{FI}(P)$ carries a representation of $\mf{g}'$.  Let $b \in \Gamma$ and $(F,I) \in \mc{FI}(P)$.  By \ref*{(iii)} we have $H_b.\la F,I \ra = X_bY_b.\la F,I \ra - Y_bX_b.\la F,I \ra$.  First suppose $b$ is the color of a minimal element $y$ of $F$.  With EC in mind, we have $Y_bX_b.\la F,I \ra = Y_b.\la F-y,I+y \ra = \la F,I \ra$.  By EC and ND we see $I$ has no maximal element of color $b$, so $X_bY_b.\la F,I \ra = 0$.  Thus $H_b . \la F,I \ra = -\la F,I \ra$.  A dual argument shows that $H_b . \la F,I \ra = +\la F,I \ra$ if $b$ is the color of a maximal element of $I$.  If $b$ is neither the color of a maximal element of $I$ nor of a minimal element of $F$, then both $X_bY_b.\la F,I \ra$ and $Y_bX_b.\la F,I \ra$ vanish.  Thus $H_b.\la F,I \ra = 0$, so \ref*{(iv)} holds.  This also shows that the eigenvalue set of the operators $\{H_a\}_{a \in \Gamma}$ is contained in $\{-1,0,1\}$, so \ref*{(i)} also holds.

Now suppose \ref*{(iv)} holds for some such diagonal operators $\{H_a\}_{a \in \Gamma}$.  By Lemma \ref{619} we know the representations of $\mf{b}'_+$ and $\mf{b}'_-$ are $X$- and $Y$-square nilpotent, respectively.  By Proposition \ref{323} we know $P$ satisfies EC and ND.  Let $b \in \Gamma$ and $(F,I) \in \mc{FI}(P)$.  To show $[X_b,Y_b] = H_b$, first suppose $Y_bX_b.\la F,I \ra \ne 0$.  Then $b$ is the color of a minimal element of $F$.  With EC in mind we have $Y_bX_b.\la F,I \ra = \la F,I \ra$.  By EC and ND we see $I$ has no maximal element of color $b$, so $X_bY_b.\la F,I \ra = 0$.  Thus $[X_b,Y_b].\la F,I \ra = -\la F,I \ra = H_b.\la F,I \ra$, the last equality following from \ref*{(iv)}.  Next suppose $X_bY_b.\la F,I \ra \ne 0$.  A dual argument obtains $[X_b,Y_b].\la F,I \ra = +\la F,I \ra = H_b.\la F,I \ra$.  Finally suppose $Y_bX_b.\la F,I \ra = 0$ and $X_bY_b.\la F,I \ra = 0$.  Then $b$ is neither the color of a minimal element of $F$ nor of a maximal element of $I$.  Hence $[X_b,Y_b].\la F,I \ra = 0 = H_b.\la F,I \ra$, again using \ref*{(iv)} for the last equality.  Thus \ref*{(iii)} holds.

We have now shown the equivalence of \ref*{(iv)}, \ref*{(iii)}, and \ref*{(i)}.  Corollary \ref{627} established the equivalence of \ref*{(ii)}, \ref*{(iv)}, and \ref*{(v)}.

This proof showed that the diagonal operators $\{H_a\}_{a \in \Gamma}$ satisfying \ref*{(iii)} also satisfy the conditions of \ref*{(iv)}.  Hence the diagonal operators $\{H_a\}_{a \in \Gamma}$ for \ref*{(i)} also satisfy \ref*{(iv)}.  Corollary \ref{627} showed that the $\mu$-diagonal operators are the unique operators satisfying the conditions of \ref*{(ii)} or \ref*{(iv)}.
\ep

\section{Main results; $\boldsymbol{\Gamma}$-colored $\boldsymbol{d}$-complete and $\boldsymbol{\Gamma}$-colored minuscule posets}\label{668}

We say a locally finite poset $P$ (of any cardinality) colored with a finite simple graph $\Gamma$ is a \emph{$\Gamma$-colored $d$-complete} poset if it satisfies EC, NA, AC, I2A, and Mx1GA with respect to $\Gamma$.  Table \ref{245} explains these abbreviations and indexes the locations of their definitions.  Figure \ref{289} displayed a $\Gamma$-colored $d$-complete poset.  Any nonempty filter of the poset displayed in Figure \ref{134} is $\Gamma$-colored $d$-complete.  We say $P$ is a \emph{$\Gamma$-colored minuscule} poset if it additionally satisfies Mn1LA.  Figure \ref{134} displayed a $\Gamma$-colored minuscule poset. 
Definition \ref{MinReps} provided the definitions of upper $P$-minuscule representations of $\mf{b}'_+$ and $P$-minuscule representations of $\mf{g}'$.
The following statement summarizes our two main results, Theorems \ref{614} and \ref{653}:

\begin{Thm}\label{672} 
Let $P$ be a poset whose elements are colored by the nodes of a finite simple graph $\Gamma$.  Let $\mc{FI}(P)$ be the lattice of filter-ideal splits of $P$.
\be[(a),nosep]
\item The lattice $\mc{FI}(P)$ carries an upper $P$-minuscule representation of $\mf{b}'_+$ if and only if $P$ is a $\Gamma$-colored $d$-complete poset.
\item The lattice $\mc{FI}(P)$ carries a $P$-minuscule representation of $\mf{g}'$ if and only if $P$ is a $\Gamma$-colored minuscule poset.
\ee
\end{Thm}

We provide two closing remarks for context.

\begin{Rmk}\label{ClassifySummary}
In \cite[Ch. 8]{Str} we classified the $\Gamma$-colored $d$-complete and $\Gamma$-colored minuscule posets.  
These results will also appear in \cite{ProStr}.
By Theorem \ref{672}, this enables us to classify all upper $P$-minuscule and $P$-minuscule representations.  
The finite and infinite connected $\Gamma$-colored minuscule posets are respectively the connected colored minuscule posets of Proctor \cite{Pro4} and the connected full heaps of Green \cite{Gre}.  The finite connected $\Gamma$-colored $d$-complete posets are the connected dominant minuscule heaps of Stembridge \cite{Ste} (the colored $d$-complete posets of Proctor \cite{Pro1} in the simply laced case).  The connected infinite $\Gamma$-colored $d$-complete posets are filters of the connected full heaps.  
Arbitrary $\Gamma$-colored $d$-complete and $\Gamma$-colored minuscule posets are direct sums of connected ones, with their Dynkin diagrams being the disjoint unions of the respective connected Dynkin diagrams coloring the posets making up these direct sums.
\end{Rmk}

\begin{Rmk}\label{GreenComparison}
Let $\mathfrak{g}'$ be a derived Kac--Moody algebra with Dynkin diagram $\Gamma$.  Let $P$ be a full heap colored by $\Gamma$.  Green defined an ideal $I$ of $P$ to be \emph{proper} if for every color $a \in \Gamma$ it satisfies $P_a \cap I \ne \emptyset$ and $P_a \cap I \ne P_a$.  Let $\mathcal{B}$ be the set of splits whose ideals are proper and let $V_{\mathcal{B}}$ be the free complex vector space on $\mathcal{B}$. 
\begin{Thm}[Theorem 4.1.6(i) of \cite{Gre}]\label{4.1.6(i)}
Let $P$ be a full heap colored by $\Gamma$.  Then $V_{\mathcal{B}}$ has the structure of a $\mathfrak{g}'$-module such that for every $a \in \Gamma$, the action of $x_a$ is given by $X_a$, the action of $y_a$ is given by $Y_a$, and the action of $h_a$ is given by the operator $H_a$ satisfying the condition in Corollary \ref{627}(ii).
\end{Thm}
\ni Suppose $P$ is a full heap over a simply laced Dynkin diagram $\Gamma$.  As noted above, the heap $P$ is also a $\Gamma$-colored minuscule poset.  
Both Theorem \ref{4.1.6(i)} and the sufficient direction of Theorem \ref{672}(b) provide a representation of $\mathfrak{g}'$.  These two representations differ only slightly when $P$ is connected.
In this case, the ideals $\emptyset$ and $P$ are the non-proper ideals of $P$, and so
$\langle \mathcal{FI}(P) \rangle = V_{\mathcal{B}} \oplus \mathbb{C} \langle P,\emptyset \rangle \oplus \mathbb{C} \langle \emptyset, P \rangle$.  Here $\mathbb{C}\langle P, \emptyset \rangle$ and $\mathbb{C}\langle \emptyset, P \rangle$ are trivial representations 
corresponding to the trivial components $\{(P,\emptyset)\}$ and $\{(\emptyset,P)\}$ of $\mathcal{FI}(P)$.
\end{Rmk}

\vs

\ni \tb{Acknowledgements:} I would like to thank my thesis advisor, Robert A. Proctor, for many helpful suggestions regarding notation, terminology, and exposition.  I would also like to thank the referee for many helpful remarks and suggestions.

\vs 

\begin{spacing}{.80} 

\end{spacing}

\end{spacing}

\begin{thebibliography}{KlRa}


\bibitem[BuSa]{BuSa}Buch, A., Samuel, M., $K$-theory of minuscule varieties, J. Reine Angew. Math. \tb{719}, 133-171 (2016).

\bibitem[Car]{Car}Carrell, J., Vector fields, flag varieties, and Schubert calculus, in Proc. Hyderabad Conf. on Alg. Groups, S. Ramanan ed., 23-57, Manoj Prakashan (1991).

\bibitem[Do1]{Don}Donnelly, R.G., Extremal properties of bases for representations of semisimple Lie algebras, J. Alg. Combin. \tb{17}, 255-282 (2003).

\bibitem[Do2]{Don2}Donnelly, R.G., Poset models for Weyl group analogs of symmetric functions and Schur functions, \href{https://arxiv.org/abs/1811.10643}{arXiv:1811.10643}.

\bibitem[Gr1]{Gre1}Green, R.M., Full heaps and representations of affine Kac--Moody algebras, Int. Elec. J. of Alg. \tb{2}, 138-188 (2007).

\bibitem[Gr2]{Gre2}Green, R.M., Full heaps and representations of affine Weyl groups, Int. Elec. J. of Alg. \tb{3}, 1-42 (2008).

\bibitem[Gr3]{Gre}Green, R.M., Combinatorics of Minuscule Representations, Cambridge University Press (2013).

\bibitem[Ha1]{Hag1}Hagiwara, M., Minuscule heaps over simply-laced, star-shaped Dynkin diagrams.  Presented at 14th Inter. Conf. on
Formal Power Series and Algebraic Combins. (FPSAC 2002, Melbourne), \url{http://www-igm.univ-mlv.fr/~fpsac/FPSAC02/ARTICLES/Manabu.pdf}.

\bibitem[Ha2]{Hag2}Hagiwara, M., Minuscule heaps over Dynkin diagrams of type $\tilde{A}$. Elec. J. Combin. \tb{11} (2004) \#R3.


\bibitem[Kac]{Kac}Kac, V.G., Infinite dimensional Lie algebras, Cambridge University Press (1990).

\bibitem[KlRa]{KleRam}Kleshchev, A., Ram, A., Homogeneous representations of Khovanov-Lauda algebras, J. Euro. Math. Soc. \tb{12}, 1293-1306 (2010).


\bibitem[McG]{McD}McGregor-Dorsey, Z.S., Some properties of full heaps, Ph.D. Thesis, University of Colorado Boulder (2013).

\bibitem[Pr1]{Pro4}Proctor, R.A., Bruhat lattices, plane partition generating functions, and minuscule representations, Euro. J. of Combin. \tb{5}, 331-350 (1984).

\bibitem[Pr2]{Pro2}Proctor, R.A., Dynkin diagram classification of $\lambda$-minuscule Bruhat lattices and of $d$-complete posets, J. Alg. Combin. \tb{9}, 61-94 (1999).

\bibitem[Pr3]{Pro1}Proctor, R.A., Minuscule elements of Weyl groups, the numbers game, and $d$-complete posets, J. Algebra \tb{213}, 272-303 (1999).

\bibitem[Pr4]{Pro5}Proctor, R.A., $d$-Complete posets generalize Young diagrams for the jeu de taquin property,	\href{https://arxiv.org/abs/0905.3716}{arXiv:0905.3716}.

\bibitem[Pr5]{Pro3}Proctor, R.A., $d$-Complete posets generalize Young diagrams for the hook product formula: Partial presentation of proof, RIMS K\^{o}ky\^{u}roku \tb{1913}, 120-140 (2014).

\bibitem[PrSc]{ProScop}Proctor, R.A., Scoppetta, L. M., $d$-Complete posets: Local structural axioms, properties, and equivalent definitions, Order (2018), \url{https://doi.org/10.1007/s11083-018-9473-4}.

\bibitem[PrSt]{ProStr}Proctor, R.A., Strayer, M. C., Classifications of colored $d$-complete and minuscule posets, infinite and finite, in preparation.

\bibitem[Sta]{Sta}Stanley, R.P., Enumerative Combinatorics, Vol. 1, Cambridge University Press (2012).

\bibitem[Ste]{Ste}Stembridge, J.R., Minuscule elements of Weyl groups, J. Algebra \tb{235}, 722-743 (2001).

\bibitem[Str]{Str}Strayer, M.C., Characterization and classification of minuscule Kac--Moody representations built from colored posets, Ph.D. Thesis, University of North Carolina at Chapel Hill (2019).

\bibitem[Wil]{Wil}Wildberger, N.J., A combinatorial construction for simply-laced Lie algebras, Adv. Appl. Math. \tb{30}, 385-396 (2003).

\end{thebibliography}
\end{document}